\documentclass[onefignum,onetabnum]{siamonline250211}

\usepackage{lipsum}
\usepackage{amsfonts}
\usepackage{graphicx}
\usepackage{epstopdf}
\usepackage{algorithmic}
\ifpdf
  \DeclareGraphicsExtensions{.eps,.pdf,.png,.jpg}
\else
  \DeclareGraphicsExtensions{.eps}
\fi

\usepackage{enumitem}
\setlist[enumerate]{leftmargin=.5in}
\setlist[itemize]{leftmargin=.5in}

\newsiamremark{remark}{Remark}
\newsiamremark{hypothesis}{Hypothesis}
\crefname{hypothesis}{Hypothesis}{Hypotheses}
\newsiamthm{claim}{Claim}
\newsiamremark{fact}{Fact}
\crefname{fact}{Fact}{Facts}

\headers{An Example Article}{N. Metya}

\title{An Example Article\thanks{Submitted to the editors \today.
}}

\author{Nilava Metya\thanks{Rutgers - the State University of New Jersey, New Brunswick, NJ 
  (\email{nilava.metya@rutgers.edu}).}
  }

\usepackage{amsopn}

\usepackage{cite}
\usepackage{listings}
\usepackage{amssymb}
\usepackage{mathtools}
\usepackage{tikz}
\usetikzlibrary{positioning,fit,backgrounds,calc}
\usepackage{tikz-cd}
\usetikzlibrary{decorations.pathmorphing,nfold}
\usetikzlibrary{matrix}
\tikzset{
  between/.style 2 args={
    to path={
      ($(\tikztostart)!#1!(\tikztotarget)$)
      -- ($(\tikztostart)!#2!(\tikztotarget)$)
      \tikztonodes
    }
  }
}

\crefname{question}{question}{questions}
\Crefname{question}{Question}{Questions}

\headers{ML Degrees via One-Edge Chordal Completion}{N. Metya}
\title{Maximum Likelihood Degrees via One-Edge Chordal Completion
\thanks{
\today.
}
}
\author{Nilava Metya\thanks{Rutgers--The State University of New Jersey, New Brunswick, NJ (\email{nilava.metya@rutgers.edu}).}}
\usepackage{comment}
\usepackage{listings}

\lstdefinelanguage{Macaulay2}{
 morekeywords={restart,ZZ,QQ,ideal,matrix,det,saturate,degree,dim,class,
  apply,select,toList,submatrix,numRows},
 sensitive=true,
 morecomment=[l]--,
 morestring=[b]"
}
\ifpdf
\hypersetup{
 pdftitle={Maximum Likelihood Degrees via One-Edge Chordal Completion},
 pdfauthor={Nilava Metya}
}
\fi

\newcommand{\fa}{~\forall~}
\newcommand{\sett}{\coloneqq}
\newcommand{\cM}{\mathcal M}
\newcommand{\cE}{\mathcal E}
\newcommand{\cB}{\mathcal B}
\newcommand{\cN}{\mathcal N}

\newcommand{\R}{\mathbb R}
\newcommand{\K}{\mathbb K}
\newcommand{\C}{\mathbb C}
\newcommand{\cC}{\mathcal C}
\renewcommand{\d}{\mathrm d}
\newcommand{\cS}{\mathcal S}
\newcommand{\dsb}{\text{DeSiCluB}}
\newcommand{\abs}[1]{\left|#1\right|}

\DeclareMathOperator{\Tr}{Tr}
\DeclareMathOperator{\cof}{Cof}
\DeclareMathOperator{\MLdeg}{MLdeg}
\DeclareMathOperator{\bl}{Bl}
\newcommand{\lb}{\left[}
\newcommand{\rb}{\right]}
\newcommand{\lp}{\left(}
\newcommand{\rp}{\right)}
\newcommand{\set}[1]{\ensuremath{\left\{#1\right\}}}
\newcommand{\st}{%
 \nonscript\;
 \ifnum\currentgrouptype=16
  \;\middle|\;
 \else
  \;|\;
 \fi
}
\newcommand{\cO}{\mathcal O}
\renewcommand{\it}[1]{{\color{cyan}#1}}

\newtheorem{question}{Question}
\newsiamremark{conjecture}{Conjecture}
\newtheorem{example}{Example}

\begin{document}

\maketitle

\begin{abstract}
The maximum likelihood degree of a Gaussian graphical model measures the algebraic complexity of its likelihood equations. This paper studies maximum likelihood degrees of Gaussian graphical models through chordal completion. When a graph becomes chordal after one missing edge is added, its maximum likelihood degree is determined by the maximal cliques and clique-tree separators containing that edge. This recovers Uhler's formula for complete bipartite graphs with one part of size two, and extends it to deleted-spine clique books, whose degree is unchanged when a page is thickened to a clique. For several fill edges, the same method gives a product formula when no clique or separator bag contains two fill edges. We also determine four classes of vertices at which clique blow-up preserves the degree. Two follow from known results, and two are established here.
\end{abstract}



\section{Introduction}

Gaussian graphical models have been extensively used in the sciences to encode conditional independence. One can associate to a graph a Gaussian model in which nonedges encode conditional independence by imposing zero entries in its concentration matrix. For an $n-$dimensional Gaussian random variable $X\sim \cN(0,\Sigma)$ it is known that $X_{i}$ and $X_{j}$ are independent conditional on the data $\set{X_{k} \st k\in[n], k\ne i,j}$ if and only if $(\Sigma^{-1})_{ij} = 0$. This is the key motivation to define a Gaussian graphical model in the following way.

\begin{definition}[Gaussian graphical model]
Let $G = (V=[n],E)$ be a finite undirected simple graph. The \it{Gaussian graphical model on $G$} comprises all centered nondegenerate Gaussian random vectors $X$ with covariance matrix $\Sigma$ such that $X_{i} \perp X_{j} \st X_{[n]\smallsetminus\set{i,j}}$ for all non-adjacent vertices $i,j\in V$. One can identify this with the set of matrices
\[\cM_{G} = \set{K\in \R^{n\times n} \st K\succ 0, K=K^{\top}, K_{ij} = 0 \fa ij\notin E \text{ and } i\ne j}.\]
\end{definition}
This stems from various studies in literature including Dempster's selection problem \cite{Dempster1972} and the Gaussian Markov model over a graph \cite{SpeedKiiveri1986,Whittaker1990}. We are interested in the problem of Maximum Likelihood Estimation on this model for observed covariances on $\mathcal O_{G}\sett \set{vv\st v\in V}\cup E$. Precisely, given observed covariance data $S=(S_{ij})_{ij \in \cO_G}$, the statistical maximum likelihood problem is to maximize $\ell_S(K)=\log\det K-\Tr(SK)$ over the real positive-definite model $\cM_G$. Statistically, maximizing the log likelihood means finding a suitable $K$ which best explains the observed data assuming the underlying model is the Gaussian graphical model. To study the algebraic complexity of this problem, we complexify the critical-point equations and work on
\[\set{K\in\C^{n\times n}\st K=K^\top, \det K\ne0, K_{ij}=0 \fa ij\notin\cO_G}.\]
Here $\log\det K$ is interpreted locally, and its differential is
independent of the chosen branch. The number of complex critical points, counted with multiplicity for generic $S$, is a constant and is called the \it{maximum likelihood degree of $G$}, or $\MLdeg(G)$ in short. 
This invariant was introduced and developed in the work of \cite{HostenKhetanSturmfels2005,CataneseHostenKhetanSturmfels2006}; see also \cite{DrtonSturmfelsSullivant2009,Sullivantalgstat,HuhSturmfels2014}.

These equations precisely are $K_{ij} = 0 \fa ij\notin \mathcal O_{G}$ and $(K^{-1})_{ij} = S_{ij}\fa ij\in \mathcal O_{G}$. Equivalently, in covariance coordinates $\Sigma=K^{-1}$, one fixes $\Sigma_{ij}=s_{ij}$ for $i=j$ or $ij\in E$ and uses the nonedge entries as variables. Since $(\Sigma^{-1})_{ij}= \cof_{ji}(\Sigma)/\det\Sigma$, the likelihood equations are the corresponding cofactor equations, saturated by $\det\Sigma$. The exact generic quotient-algebra formulation is given in \eqref{eq:likelihoodequations}.

Chordal, or decomposable graphs, form a special class of Gaussian graphical models in the following sense. If $G$ is chordal, then given any generically observed covariances $(\Sigma_{ij})_{ij\in \mathcal O_{G}}$, one can find a unique rational matrix $K=\hat \Sigma^{-1}$ with $K=K^{\top}, K_{ij} = 0\fa ij\notin \cO_{G}$ satisfying $\hat \Sigma_{ij} = \Sigma_{ij}\fa ij\in \cO_{G}$. There is also a connection with matrix completions which we shall not discuss and refer the reader to \cite{GroneJohnsonSaWolkowicz1984}. Chordal completions are also useful computationally for nonchordal covariance selection; see \cite{DahlVandenbergheRoychowdhury2008}.

\begin{example}[Chordal graph]\label{ex:chordalcomputation}
Consider the following graph, with its covariance matrix, with observed covariance data $a,b,c,d$ as indicated and $s_{1},s_{2},s_{3},s_{4}$ on the diagonals.
\[\begin{tikzcd}
	{1\bullet} & {\bullet 2} \\
	{3\bullet} & {\bullet 4}
	\arrow["a", no head, from=1-1, to=1-2]
	\arrow["c"', no head, from=1-1, to=2-2]
	\arrow["b", no head, from=1-2, to=2-2]
	\arrow["d"', no head, from=2-1, to=2-2]
\end{tikzcd}, \qquad 
\Sigma(\alpha,\beta) = 
\begin{bmatrix}
s_{1} & a & \alpha & c\\
a& s_{2} & \beta & b\\
\alpha & \beta & s_{3} & d\\
c & b & d & s_{4}
\end{bmatrix}
\]
where $\alpha,\beta$ are unobserved. The constraints we need are $(\Sigma^{-1})_{13}=(\Sigma^{-1})_{23}=0$. Away from the zeros of $\det \Sigma(\alpha,\beta)$, the equations are $\cof_{13} = \cof_{23} = 0$. Note that both these cofactors are affine-linear in $\alpha,\beta$. Therefore, one can simultaneously solve these linear equations in $\alpha,\beta$ away from a vanishing condition which is polynomial in the observed data $\set{a,b,c,d,s_{1},s_{2},s_{3},s_{4}}$.
\end{example}

\begin{example}[Nontrivial ML degree]\label{ex:4cycledegree}
Consider the cyclic graph $C_{4}$, with its covariance matrix, with observed covariance data $a,b,c,d$ as indicated and $s_{1},s_{2},s_{3},s_{4}$ on the diagonals.
\[\begin{tikzcd}
	{1\bullet} & {\bullet 2} \\
	{4\bullet} & {\bullet 3}
	\arrow["a", no head, from=1-1, to=1-2]
	\arrow["d"', no head, from=1-1, to=2-1]
	\arrow["b", no head, from=1-2, to=2-2]
	\arrow["c"', no head, from=2-1, to=2-2]
\end{tikzcd}, \qquad 
\Sigma(\alpha,\beta) = 
\begin{bmatrix}
s_{1} & a & \alpha & d\\
a& s_{2} & b & \beta\\
\alpha & b & s_{3} & c\\
d & \beta & c & s_{4}
\end{bmatrix}
\]
where $\alpha,\beta$ are unobserved. The constraints we need are $(\Sigma^{-1})_{13}=(\Sigma^{-1})_{24} = 0$. The equations of the numerator vanishing are $\cof_{13} = \cof_{24} = 0$. Now,
\begin{align*}
\cof_{13}(\alpha,\beta) &= \alpha\beta^{2} -\alpha s_{2}s_{4} - \beta(bd + ac) + abs_{4}+cds_{2},\\
\cof_{24}(\alpha,\beta) &= \alpha^{2}\beta -\alpha (bd+ac) -\beta s_{1}s_{3} + ads_{3} + bcs_{1}.
\end{align*}

Homogenize the above polynomials to total degree three by the extra variable $\gamma$. Bezout's theorem gives nine solutions. But one must discount for the roots at infinity, namely those lying on $\gamma=0$, and those of the denominator. The corresponding equations are $\alpha\beta^{2}=\alpha^{2}\beta=0$, whence the roots at infinity are $[1:0:0], [0:1:0]$. But each of these occur with multiplicity two, so we discount a total of four roots. This gives five roots in the affine part $\set{\gamma\ne 0}$ of $\mathbb P^{2}$. For the latter, we claim that there is no common root of the cofactors and the denominator and leave it to the reader to verify this for generic entries. 


\end{example}
Example~\ref{ex:chordalcomputation} is a toy example of the fact that the $\MLdeg$ of a chordal graph is one. In fact, the converse was proved by \cite{AmendolaDinuMichalekVodicka2024}. Two more known cases of closed form formulas are Uhler's formula \cite[Theorem~4.3]{Uhler_mle_ggm} and the cycle formula conjectured in \cite{DrtonSturmfelsSullivant2009} and proved in \cite[Theorem~1.3]{DinuVodicka2025}:
\begin{align}
 \MLdeg(K_{2,m})&=2m+1,& m\geq2.
 \tag{Uhler's formula}\label{eq:uhlerformula}\\
 \MLdeg(C_n)&=(n-3)2^{n-2}+1,& n\geq4
 \tag{DVDSS formula}.\label{eq:dvdss}
\end{align}

In this paper, we first focus on obtaining a combinatorial formula of the $\MLdeg$ of a \it{one-fill chordal graph}. Such a graph is one obtained by deleting an edge from a chordal graph. We prove this using the same proof strategy of \ref{eq:uhlerformula}. Chordal graphs are special because of the tractability of the likelihood equations, popularly known as the clique separator formula of Lauritzen \cite[Proposition~4.5]{lauritzenlecture} mentioned here in Proposition~\ref{prop:chordal-formula}.

\begin{theorem}[One-fill chordal ML-degree formula]\label{thm:one-fill}
Let $H$ be a chordal graph, $uv \in E(H)$, and  $G=H\smallsetminus uv$. Fix a clique tree $T$ of $H$. Let $\cC_{uv}(H)$ be the set of maximal cliques of $H$ containing both $u,v$, and let $\cS_{uv}(H,T)$ be the set of distinct clique-tree separator bags containing both $u,v$. Then, independent of the chosen clique tree $T$,
\begin{equation}\label{eq:off}\tag{one-fill formula}
 \MLdeg(G) = 2\lp\abs{\cC_{uv}(H)}+\abs{\cS_{uv}(H,T)}\rp-1.
\end{equation}
\end{theorem}

\begin{example}
Consider $C_{4}$ from \Cref{ex:4cycledegree} with $\MLdeg=5$. Notice that this graph is also one-fill chordal. Indeed joining vertices $u=1,v=3$ makes it chordal. Let $H$ be the resulting graph after joining $1,3$. Its maximal cliques are $\cC(H)= \set{ \set{1,2,3} , \set{1,3,4} }$. The only clique tree of $H$ is the tree with these vertices and one edge between them. The running intersection property is trivial to verify. Therefore $\abs{\cC_{1,3}(H)}=2, \abs{\cS_{1,3}(H,T)} = 1$, which gives $\MLdeg(G) = 2(2+1)-1 = 5$ as we earlier checked.
\end{example}

For $m\ge 2$, $G = K_{2,m}$ is indeed a one-fill chordal graph where joining the two vertices on the $2-$side of $G$ yields a chordal graph $H$ with $m$ clique bags and one distinct separator bag, thus recovering \ref{eq:uhlerformula}. Observe that replacing the vertices on the $m-$ side of $G$ with arbitrary cliques, which we refer to as blow-up later, preserves the relevant form and degree of the equations. This gives the deleted-spine clique-book formula in \Cref{cor:book} namely the following. If $m\geq2$ cliques $\set{u,v}\cup A_i$ share only the edge $uv$ and then this edge is deleted, the ML degree is $2m+1$, independent of  $\abs{A_{i}}$. The same method gives a product formula in the case when many edges are deleted from a chordal graph, but with certain independence property, which we call the bag-noninteraction hypothesis.

Finally we investigate \cref{qs:clique-blow-up,qs:special-vertex} about operations preserving the $\MLdeg$. One such classical operation is given by the suspension identity \cite[Theorem~4.11]{SturmfelsUhler2010}. We focus on the blow-up construction wherein one replaces a vertex by a clique and joins every neighbor of the vertex with the vertices of the clique. 

\begin{definition}[Clique blow-up]\label{def:clique-blow-up}
Let $q\geq1$. The \it{clique blow-up $\bl_{w}^{q}(G)$ of $G$ at vertex $w$} replaces $w\in V(G)$ by a clique $Q_w \simeq K_{q}$ of $q$ vertices. More precisely,
\begin{align*}
V(\bl_{w}^{q}(G))&= (V(G)\smallsetminus\set{w})\sqcup Q_{w},\\
E(\bl_{w}^{q}(G))&= E(G\smallsetminus\set{w}) \cup\binom{Q_{w}}{2} \cup\set{az\st a\in\cN_{w,G}, z\in Q_w}.
\end{align*}
where $\abs{Q_{w}} = q$. The vertices of $Q_w$ are called \it{twins}.
\end{definition}

It turns out that blowing-up only special vertices preserves the $\MLdeg$. Four such special classes are identified, as summarized in the table below, though the list is not exhaustive. \cite[Lemma 4.2, Theorem~4.11]{SturmfelsUhler2010} cover simplicial and universal vertex blow-ups. \Cref{thm:one-fill} covers blow-up of a vertex away from a fill edge of a one-fill chordal graph. A genuinely new result \Cref{thm:degree-two-blow-up} shows the invariance under blow-up of a degree-two vertex. 

\begin{equation}\label{eq:table-of-blowups}
\begin{array}{c|c|c}
\text{type of vertex $w$}&\text{condition/meaning}&\text{proof}\\
\hline
\text{simplicial}&\cN_{w,G}\text{ is a clique}
 &\Cref{prop:simplicial-blow-up}\\
\text{universal}&\cN_{w,G}=V(G)\smallsetminus\set{w}
 &\Cref{prop:universal-blow-up}\\
\text{away from a fill edge}&G+uv\text{ is chordal and }w\notin\set{u,v}
 &\Cref{prop:onefill-blow-up}\\
\text{degree two}&\deg_G(w)=2
 &\Cref{thm:degree-two-blow-up}
\end{array}\end{equation}

\paragraph{Organization} \Cref{sec:prem} discusses preliminaries and the scope of the formulas. \Cref{sec:equations} gives the likelihood equations, clique-separator formula, algebraic fill reduction and shows aspects of one-fill reduction. \Cref{sec:one-fill} proves \Cref{thm:one-fill}. \Cref{sec:several-fill} proves the factorization by interaction components and the product formula for noninteracting fill edges. \Cref{sec:blow-up} studies clique blow-up and non-exhaustively answers \cref{qs:special-vertex}. The appendices contain supplementary lemmas and the double-blow-up computation for a negative answer to \cref{qs:clique-blow-up}. 

\section{Preliminaries and perspectives}\label{sec:prem}

\subsection{Graph theoretic preliminaries}
All graphs are finite, simple, and undirected. If $G=(V,E)$ and $F\subseteq E$, then $G\smallsetminus F$ is obtained by deleting the edges in $F$ without deleting any vertex. If $F$ is a set of nonedges of $G$, then $G+F$ is obtained by adding the edges in $F$. We abbreviate $G\smallsetminus\set{uv}, G+\set{uv}$ to $G\smallsetminus uv, G + uv$ respectively. $K_t$, $K_{m,n}$, and $P_n$ will respectively denote the complete graph on $t$ vertices, the complete bipartite graph with  bipartition $[m]\sqcup[n]$, and the path graph on $n$ vertices. For $v\in V(G)$, let $\cN_{v,G}$ be the set of neighbors of $v$ in $G$. Sometimes we simply write $\cN_{v}$ when the graph $G$ is clear from context. 

A \it{clique} is a set of pairwise adjacent vertices, and it is \it{maximal} if it is not properly contained in another clique. The set of maximal cliques of $G$ is denoted by $\cC(G)$. If $B\subseteq V(G)$ and $M \in \R^{B\times B}$ is a matrix indexed by $B$, then $M ^{0} \in \R^{V(G)\times V(G)}$ denotes its embedding into a matrix indexed by $V(G)$, with zero entries outside $B\times B$. A graph is \it{chordal} or \it{decomposable} if every cycle of length at least four has a chord. Equivalently, it has no induced cycle of length at least four.

\begin{definition}[Clique tree, bag, and separator]\label{def:clique-tree}
Let $H$ be chordal. A \it{clique tree} of $H$ is a tree $T$ with vertex set $\cC(H)$ such that, for every $w\in V(H)$, the set $\set{C\in\cC(H)\st w\in C}$ of maximal cliques of $H$ containing $w$ induces a connected subtree of $T$. For $e=CC'\in E(T)$, the set $B_e=C\cap C'$ is called the \it{separator bag} associated to $e$. Each vertex of $T$, namely each maximal clique of $H$, is called a \it{clique bag}.
\end{definition}

The above requirement `for every $w\in V(H)$, the set of maximal cliques of $H$ containing $w$ form a connected subtree of a clique tree $T$' is called the \it{running intersection property}. It is well known, for example from \cite[Chapter~2]{Lauritzen1996}, that every chordal graph admits a clique tree.

\begin{remark}[Connectedness]
For a disconnected chordal graph with $r$ components, choose a clique tree for each component and connect those trees by $r-1$ cross-component edges so that the result is a tree. Every new separator is empty. We adopt $\det(\Sigma_\varnothing) = 1$ and $(\Sigma_\varnothing^{-1}) ^{0}=0$. Thus all formulas below apply without any connectedness assumption. Separator bags are counted with multiplicity in sums indexed by $E(T)$, but $\cS_f(H,T)$ below denotes a set of distinct subsets of $V(H)$.
\end{remark}

We now discuss some important properties of the clique tree. For $f=uv\in E(H)$, let
\[
\cC_f(H) \sett \set{C\in\cC(H)\st\set{u,v}\subseteq C}, \qquad  \cS_f(H,T)\sett\set{B_e\st e\in E(T), \set{u,v}\subseteq B_e},
\]
where the second expression is a set of distinct separator bags. We also write $\cC_{uv}$ and $\cS_{uv}$. No separator $C\cap C'$ of two distinct maximal cliques can itself be a maximal clique. Indeed, otherwise it would be one of $C$ or $C'$ whence one maximal clique would be properly contained in the other. Consequently, the following union is disjoint:
\[\cB_f(H,T)=\cC_f(H)\sqcup\cS_f(H,T).\]
For $B\in\cB_f(H,T)$, define
\begin{equation}\label{eq:net-coefficient}
\lambda_B^{(f)}= \abs{\set{C\in\cC_f(H)\st C=B}} - \abs{\set{e\in E(T)\st B_e=B}}.
\end{equation}
Thus a clique bag has coefficient $1$, whereas a separator bag has coefficient equal to the negative of its positive multiplicity among the relevant edges of $T$.

If $G,H$ have the same vertex set then the elements of $F\sett E(H)\smallsetminus E(G)$ are called the \it{fill edges}. The covariance matrices of both these graphs have the same size. The unobserved (with respect to $G$) covariance entry corresponding to $f=uv\in F$ will be denoted by $x_{f}=\Sigma_{uv}$, or simply $x$ when $F$ is the singleton $\set{uv}$.

\begin{lemma}[Clique tree properties]\label{lem:properties}
Let $H$ be a chordal graph and $T$ a clique tree of $H$. Recall that the vertex set of $T$ is the set $\cC(H)$ of maximal cliques in $H$. The following are true.
\begin{enumerate}[label=(\roman*)]
\item \label{path-property} For $C,C'\in \cC(H)$ if $w\in C\cap C'$ then every clique on the unique path in $T$ between $C,C'$ contains $w$.
\item \label{uv-subtree} For $f= uv \in E(H)$, the maximal cliques containing $u$ and $v$ induce a nonempty subtree $T_{f}$ of $T$. Its edges are exactly the edges $e = CC'\in E(T)$ for which $\set{u,v}\subseteq B_{e} = C\cap C'$.
\item \label{sum-net-coefficients} For every $f\in E(H)$, $\sum_{B \in \cB_{f}(H,T)} \lambda_{B}^{(f)}=1$.
\end{enumerate}
\end{lemma}
\begin{proof}
\begin{enumerate}[label=(\roman*)]
\item By the running-intersection property, the maximal cliques containing $w$ induce a connected subgraph of the tree $T$. A connected subgraph of a tree that contains $C$ and $C'$ contains their unique connecting path.
\item The maximal cliques containing $u$, and respectively $v$, each induce a subtree of $T$. Their intersection is nonempty because the edge $uv$ lies in a maximal clique. The intersection of two nonempty intersecting subtrees of a tree is connected, so its vertex set is that of a subtree $T_f$. An edge $e=CC'$ belongs to $T_f$ precisely when both $C$ and $C'$ contain $u$ and $v$, equivalently when $B_e=C\cap C'$ contains $\set{u,v}$.
\item The sum, over $B$, of the first terms in \eqref{eq:net-coefficient} gives $\abs{V(T_f)}$. Similarly the second terms give $\abs{E(T_f)}$. Since $T_{f}$ is a connected tree and the above sum is the difference between the number of vertices and edges of $T_{f}$, the mentioned sum must be one.
\end{enumerate}
\end{proof}

For studying invariance of $\MLdeg$ under combinatorial operations, a special type of chordal graph that pops up is the complete graph $K_{n}$. \cite{SturmfelsUhler2010} successfully identified two such operations which we shall describe here briefly.

\begin{definition}[Clique sum]\label{def:clique-sum}
After relabeling if necessary, suppose $S = V(G_{1})\cap V(G_{2})$ induces a clique in both $G_{1},G_{2}$. The \it{clique sum $G_{1}\cup_{S} G_{2}$ of $G_1$ and $G_2$ along $S$} is the graph with vertices $V(G_1)\cup V(G_2)$ and edges $E(G_1)\cup E(G_2)$. 
\end{definition}

\begin{definition}[Join]\label{def:join}
The \it{join} $G_{1}*G_{2}$ of graphs $G_{1},G_{2}$ has vertices $V(G_1)\sqcup V(G_2)$ and edges $E(G_1)\cup E(G_2)\cup \set{ij\st i\in V(G_1),j\in V(G_2)}.$
\end{definition}

Intuitively, a \textit{clique sum} is a ‘gluing’ of two graphs along a common clique. One can choose different sizes or positions of this clique within each graph to produce different clique sums. Note that the clique sum along $\varnothing$ is a disjoint union. The \textit{join} is the graph obtained by taking the disjoint union of the two graphs and then adding all possible cross-edges between them. The following two established identities will be used later and also form the motivation for defining a clique blow-up in order to identify operations that help compute the $\MLdeg$.

\begin{lemma}[{\cite[Lemma~4.2]{SturmfelsUhler2010}}]\label{lem:mult-clique-sum}
$\MLdeg(G_1\cup_{K_t}G_2)=\MLdeg(G_1)\MLdeg(G_2)$.
\end{lemma}

\begin{theorem}[{\cite[Theorem~4.11]{SturmfelsUhler2010}}]\label{thm:suspension}
$\MLdeg(G*K_t)=\MLdeg(G)$.
\end{theorem}

\subsection{Usefulness of the formulas}

The organization principle for the first part of the paper is chordal completion. If $H$ is chordal and $G=H\smallsetminus uv$ where $uv\in E(H)$, then the likelihood equations of $G$ correspond to the solutions of the rational equation $K_{uv}^{H}(x) = 0$ where $x=\Sigma_{uv}$. This is proven in \Cref{lem:one-fill-reduction}.

Another central example in the first part of the paper will be the example of a \dsb. This is quite special because of the observation that `blowing up' the vertices on the $n-$side of $K_{2,n}$ does not change the structure of the equations, resulting in the same $\MLdeg$. This example also motivates the question about invariance of $\MLdeg$ under combinatorial operations on the initial graph.

\begin{definition}[Deleted-spine clique book (\dsb)]\label{def:book}
Let $n\geq1$, let $u,v$ be distinguished vertices, and let $A_{1},\cdots,A_{n}$ be pairwise disjoint nonempty sets disjoint from $\set{u,v}$. The \it{deleted-spine clique book} $G=\operatorname{Book}(A_{1},\cdots,A_{n})$, or \dsb~in short, has
\[
V(G)=\set{u,v}\sqcup\bigsqcup_{i=1}^{n} A_{i}, \qquad E(G)= \lp\bigsqcup_{i=1}^n \binom{A_{i}}{2}\rp \sqcup\set{xa\st x\in \set{u,v}, a\in \bigsqcup_{i=1}^{n} A_{i}}.
\]
The $A_{i}$ are called its \it{pages}, and the nonedge $uv$ is called its \it{deleted spine}.
\end{definition}

Looking at the structure of \dsb s, it is simply replacing each vertex on the larger side of $K_{2,n}$ with cliques which are essentially copies of the replaced vertex. The neighbor of each $i$ cannot distinguish between $i$ and each vertex in $A_{i}$, behaving like `twins'. 

For $n\geq2$, adding $uv$ to $\operatorname{Book}(A_{1},\cdots,A_{n})$ produces a chordal graph with maximal cliques $\set{u,v}\cup A_{i}$. Hence \Cref{cor:book} will give \[\MLdeg(\operatorname{Book}(A_{1},\ldots,A_{n}))=2n+1.\]
When $n=1$, the book is the clique sum of the complete graphs on $A_{1}\cup\set{u}$ and $A_{1}\cup\set{v}$ along $A_{1}$, and therefore has $\MLdeg = 1$. 

\subsection{Two perspectives on \dsb s}

\dsb s have two complementary descriptions which have already been hinted at. On one hand, they are one-fill chordal graphs; while on the other, they are clique blow-ups of $K_{2,n}$. The first description allows us to compute their $\MLdeg$ using \Cref{thm:one-fill}. The latter perspective also enables us to compute their $\MLdeg$ via the invariance of ML degree under clique blow-ups studied later in \Cref{sec:blow-up}. The first perspective along with its calculation of the $\MLdeg$ and the fact that \dsb s are clique blow-ups of $K_{2,n}$ motivates the following question.

\begin{question}\label{qs:clique-blow-up}
Is $\MLdeg(\bl_w^q(G))=\MLdeg(G)$ for every graph $G$, $w\in V(G)$, and $q\geq1$?
\end{question}

The clique condition (for blow-up) is essential in the question. Replacing a vertex by the graph with $V=\set{1,2}, E=\varnothing$ can change the degree. Indeed, doing this to a vertex in the part of size two of $K_{2,3}$ gives $K_{3,3}$, whereas $\MLdeg(K_{2,3})=7$ and $\MLdeg(K_{3,3})=57$ by \cite{Uhler_mle_ggm}. `Blowing up' by a chordal graph is also not enough and can change the degree. Indeed, `blowing up' the middle vertex of $P_3$ by $P_n$, for $n\geq3$, gives $E_2*P_n$, where $E_2$ is the edgeless graph on two vertices, and \Cref{thm:one-fill} gives $\MLdeg(E_2*P_n) = 4n-7 \neq1 = \MLdeg(P_3)$.

There are some positive answers to the above question. Chordal graphs remain chordal under clique blow-up. \Cref{prop:onefill-blow-up,thm:degree-two-blow-up} and the \ref{eq:off} prove invariance in the cases needed for \dsb s. On the other hand, because $G$ is an induced subgraph of $\bl_w^q(G)$, the result on induced-subgraph monotonicity from \cite{AmendolaDinuMichalekVodicka2024} gives $\MLdeg(\bl_w^q(G))\geq\MLdeg(G)$. Thus any counterexample must strictly increase the degree. \Cref{ex:double-blow-up} provides one constructed from $C_{4}$. The next natural problem is to identify the vertices for which blow-up invariance holds.

\begin{question}\label{qs:special-vertex}
Which vertices $w\in V(G)$ satisfy $\MLdeg(\bl_w^q(G))=\MLdeg(G) \fa q\geq1$?
\end{question}

\section{Gaussian graphical likelihood equations}\label{sec:equations}

Let $\cO_{G} \sett \set{ii \st i\in V(G)} \sqcup E(G)$ be the set of observed positions for a graph $G$. Consider \[\mathcal L_{G} \sett \set{K \in \C^{V(G)\times V(G)} \st K=K^{\top}, K_{ij}=0\fa ij\notin \cO_{G}}.\]
The open Gaussian graphical model, in covariance coordinates, is
\[\cM_{G}^{o} = \set{\Sigma \in \C^{V(G)\times V(G)} \st \det \Sigma\ne 0, \Sigma^{-1}\in \mathcal L_{G}}.\]
Thus, $\cM_{G}^{o}$ is irreducible and has dimension $\abs{\cO_{G}}$. For sufficient statistics $S = (s_{ij})_{ij\in \cO_{G}}$, the likelihood equations are \begin{align}
(\Sigma^{-1})_{ij}&=0 &\fa ij\notin \mathcal O_{G},\label{eq:zero}\\
\Sigma_{ij}&=s_{ij} &\fa ij\in\mathcal O_G.\label{eq:observed}
\end{align}

Equivalently, they form the fiber over $S$ of the projection \[
\pi_{G} : \cM_{G}^{o} \to \C^{\cO_{G}}, \qquad \Sigma \mapsto (\Sigma_{ij})_{ij\in \cO_{G}}.\]
By \Cref{lem:likelihood-projection-generically-finite} the projection  $\pi_{G}$ is dominant\footnote{recall that a morphism is \it{dominant} if its image is dense} and generically finite, and its source and target have the same dimension. The ML degree is the length of the generic fiber of $\pi_{G}$. Equivalently, if we keep the generic data symbolic and say $\K_{G} \sett \C(s_{ij} : ij\in \cO_{G})$ then the ML degree is the dimension of the generic fiber's coordinate algebra over $\K_{G}$. More precisely, let $\cE_{G} \sett \cO_{G}^{c}$ be the set of non-edges of $G$, introduce variables $z=(z_{ij})_{ij\in \cE_{G}}$ and write $\Sigma_{G}(s,z)$ for the $V(G)\times V(G)$ matrix whose entries on $\cO_{G}$ are $s$ and those on $\cE_{G}$ are $z$. Then \[\MLdeg(G) = \dim_{\K_{G}}\left(\frac{\K_{G}[z_{\cE_{G}}]}{(\cof_{ij}(\Sigma_{G}(s,z)):ij\in \cE_{G}):(\det \Sigma_{G}(s,z))^{\infty}}\right).\]
The above makes sense because $(\Sigma^{-1})_{ij} = 0 \iff \cof_{ij}(\Sigma)=0$ away from $\det \Sigma=0$. This is why we quotient $R_{G}\sett \K_{G}[z_{\cE_{G}}]$ by the ideal $I_{G} \sett (\cof_{\cE_{G}}(\Sigma_{G}(s,z))):(\det \Sigma_{G}(s,z))^{\infty}$. The following application of \Cref{lem:determinant-saturation-localization} gives an alternative description of the likelihood equations.
\begin{equation}
\label{eq:likelihoodequations}
\mathcal A_{G} \sett \frac{R_{G}}{I_{G}} = \frac{\K_{G}[z_{\cE_{G}}]}{\lp \cof_{\cE_{G}}(\Sigma_G)\rp : (\det\Sigma_G)^\infty} \stackrel{\Cref{lem:determinant-saturation-localization}}{\cong} \frac{\K_{G}[z_{\cE_{G}}][(\det \Sigma_{G})^{-1}]}{\lp \cof_{\cE_{G}}(\Sigma_G)\rp }.\end{equation}

We now recall the chordal formula from \cite[Proposition~4.5]{lauritzenlecture} and \cite[Chapter~5]{Lauritzen1996}. Intuitively it says that there is only one matrix completion when the observed data are generic for the Gaussian graphical model on a chordal graph. Further, the unique completion of the concentration matrix $K$ and its determinant is a formula in terms of a clique tree of the chordal graph. Formally, let $H$ be chordal with observations $y = (y_{ij})_{ij\in \cO_{H}}$. Consider a clique tree $T$ of $H$. For generic $y$, there is a unique completion to a non-singular matrix $K = K^{H}(y)\in\mathcal L_H$ such that $(K^{-1})_{ij} = y_{ij} \fa ij\in \cO_{H}$ and is given by the formula \begin{equation}\label{eq:chordal-formula}
 K^H(y) = \sum_{C\in\cC(H)}(y_C^{-1}) ^{0} - \sum_{e\in E(T)}(y_{B_e}^{-1})^{0}.
\end{equation}

\begin{proposition}[Chordal clique-separator formula]\label{prop:chordal-formula}
In the above setting, suppose that $y_{C}$ and $y_{B_e}$ are invertible for every $C\in\cC(H)$ and $e\in E(T)$ then there is a  unique non-singular matrix $K \in\mathcal L_H$ such that $(K^{-1})_{ij} = y_{ij}\fa ij\in \cO_{H}$. It is given by \eqref{eq:chordal-formula}. In addition,
\begin{equation}\label{eq:chordal-det-factorization}
\det K^H(y)= \frac{\prod_{e\in E(T)}\det y_{B_e}}{\prod_{C\in\cC(H)}\det y_C}.
\end{equation}
\end{proposition}

We make a minor observation about \eqref{eq:chordal-formula} which will be useful later. If $B \in \cC(H) \cup \set{B_{e}\st e\in E(T)}$ and $f=uv \in E(H)$, then $\lp(y_{B}^{-1})^{0}\rp_{f} = 0$ unless both endpoints of $f$ are contained in $B$. This is because of the embedding of $y_{B}^{-1}$, which is a matrix on $B\times B$, into the $V(H)\times V(H)$ matrix by padding with zeros. Therefore,
\begin{equation}\label{obs:padding}
\lp(y_{B}^{-1})^{0}\rp_{f} = \begin{cases}
\lp y_{B}^{-1}\rp_{f} & \text{ if } \set{u,v}\subseteq B\\
0& \text{ otherwise}. 
\end{cases}
\end{equation}

Now we move on to the setting of a graph's completion to a chordal graph. Let $H$ be chordal and $F\subseteq E(H)$. Take $G = H \smallsetminus F$. The edges in $F$ are the fill edges for $G$ to reach a chordal completion $H$. Clearly $\cO_{H} = \cO_{G} \sqcup F$. Note that this may not be the minimal chordal completion. We analyze the likelihood algebra of $G$ using $H$. The motivation for such an idea is as follows. The covariance matrix of $H$ looks like the covariance matrix of $G$ on $\cO_{G}$, and the inverse of the covariance matrix of $G$ should have a zero entry corresponding to $\cE_{H}$ and also on $F$. Recall that the likelihood algebra $\mathcal A_{G}$ of $G$ is cut out by the zero entry equations corresponding to $\cE_{G}$. So one should be able to figure out the ML degree of $G$ by first taking generic data on $F$ and $\cO_{G}$, which should give a unique concentration matrix completion for $H$, and then imposing the extra equations that the $f-$entries of this concentration matrix are all zero for $f\in F$.

Mathematically, we mean the following. Work over the generic data field $\K_{G}$. Introduce a variable $x_{f}$ for each $f\in F$. Let $y(x_{F})$ be the partial matrix on $H$ whose entries corresponding to $\cO_{G}$ are $s_{ij}$ and those corresponding to $f\in F$ are $x_{f}$. In other words, 
\[
y(x_{F})_{ij} = \begin{cases}
s_{ij} & \text{ if } ij\in \cO_{G}\\
x_{ij} & \text{ if } ij\in F
\end{cases}.
\]
Fix a clique tree $T$ of $H$ and let \[\Delta_{H}(x_{F}) = \prod_{C\in \cC(H)} \det y_{C}(x_{F})\cdot \prod_{e\in E(T)} \det y_{B_{e}}(x_{F}).\]
This gives a symbolic matrix for sufficient statistics on $H$. Then for generic values of these symbols, one can complete it to an invertible concentration matrix $K^{H}(x_{F})$, by \Cref{eq:chordal-formula}, for the Gaussian graphical model on $H$ whose inverse agrees with $(s_{ij})_{ij\in \cO_{G}}\cup(x_{ij})_{ij\in F}$ on $\cO_{H}$. We are interested in the algebra \begin{equation}
A_{H,F} \sett R_{\Delta_{H}}/\lp K^{H}_{f} : f\in F\rp.\label{eq:fill-algebra} 
\end{equation}
where $R=\K_{G}[x_{f}:f\in F]$. We will show that this is isomorphic to the likelihood algebra for $G$ over $\K_{G}$, whence $\MLdeg(G) = \dim_{\K_{G}}A_{H,F}$.

\begin{lemma}[Algebraic fill reduction]\label{lem:algebraic-fill-reduction}
Let $H, F, G, x_{F}, \Delta_{H}, K^{H}(x), R, A_{H,F}$ be as above. Then there is a $\K_{G}-$algebra isomorphism $A_{H,F}\cong \mathcal A_{G}$. 
\end{lemma}

\begin{proof}
Let $Q = \cE_{H}$. Introduce variables $z_{Q}$ on the non-edges $Q$ of $H$. Let $\Sigma(s,x,z)$ be the symmetric matrix on $V(H)=V(G)$ with entries $s_{ij}$ for $ij\in \cO_{G}$, $x_{f}$ for $f\in F$ and $z_{q}$ for $q\in Q$. Take $d=d(s,x,z)=\det \Sigma(s,x,z)$. First consider the localized likelihood algebra \[B\sett \frac{\K_{G}[x_{F}, z_{Q}, (d\Delta_{H})^{-1}]}{(\cof_{ij}(\Sigma) : ij \in \cE_{G})}.\]
By \eqref{eq:likelihoodequations}, $B\cong(\mathcal A_G)_{\Delta_H}$. Write $A=A_{H,F}$ for this proof. We first construct an explicit $\K_G$-algebra isomorphism $A \cong B$. Since $d\Delta_H$ is a unit in $B$, both $d$ and $\Delta_H$ are units in $B$. Thus the defining cofactor relations in $B$ are equivalent to the equations $(\Sigma^{-1})_{ij}=0 \fa ij\in F\cup Q=\cE_G$.

Consider the morphisms $\phi,\psi$ given by \begin{align*}
\phi:B\to A, &\qquad x_{f} \mapsto x_{f}, \qquad  z_{q} \mapsto \Sigma^{H}(x)_{q};\\
\psi:A\to B, &\qquad x_{f}\mapsto x_{f};
\end{align*}
where $\Sigma^{H}(x) = K^{H}(x)^{-1}$. By \eqref{eq:chordal-formula}, $K^H(x)$ has entries in $R_{\Delta_{H}}=\K_{G}\lb\set{x_{f}:f\in F}, \Delta_{H}^{-1}\rb$. Every clique and separator determinant is a unit in $R_{\Delta_{H}}$, so \eqref{eq:chordal-det-factorization} implies $\det K^{H}(x)$ is also a unit. Consequently, $\Sigma^{H}(x)$ has entries in $R_{\Delta_{H}}$ and its determinant is also a unit in $R_{\Delta_{H}}$.


To verify that these maps respect the corresponding relations, note that in $A$ the matrix $\Sigma^{H}(x) = K^{H}(x)^{-1}$ has the entries $y(x)$ on $\cO_{H}$ (namely $s_{ij}$ and $x_{F}$) and $K^{H}(x)$ has zero entries on $Q$. Further, $\det \Sigma^{H}$ and all factors of $\Delta_{H}$ are units in $A$ by \eqref{eq:chordal-det-factorization}. Therefore $\phi$ respects all relations and localizations. Conversely in $B$, $\Sigma(s,x,z)$ is invertible, $\Sigma(s,x,z)^{-1}$ has zero entries on $Q$ and $\Sigma(s,x,z)$ agrees with $(s_{ij})_{\cO_{G}}\cup (x_{f})_{F}$ on $\cO_{H}$. $\Delta_{H}$ remains a unit. The unique completion to a concentration matrix, by the chordal formula, gives $\Sigma(s,x,z)^{-1} = K^{H}(x)$. The defining relations of $B$ (from \eqref{eq:fill-algebra}) gives $\Sigma(s,x,z)^{-1}_{f} = 0\fa f\in F$. Therefore $\psi$ also respects all relations and localizations. Therefore both $\phi,\psi$ define morphisms.

We now verify that this is an isomorphism. $\phi\psi$ fixes every $x_f$. $\psi\phi$ fixes every $x_{f}$ and by the above discussion, maps $z_{q}$ to itself (because $\Sigma(s,x,z) = \Sigma^{H}(x)$ and chordal graphs have only one completion). This proves the isomorphism $A \cong B$.

Now by \eqref{eq:likelihoodequations}, $B \cong (\mathcal A_{G})_{\Delta_{H}}$. By \Cref{lem:Delta-unit}, the image of $\Delta_H$ is already a unit in $\mathcal A_{G}$, whence $A\cong B \cong (\mathcal A_{G})_{\Delta_{H}} \cong \mathcal A_{G}$.
\end{proof}

\begin{corollary}
For each $f$, choose $N_f\geq0$ such that $p_f=\Delta_H^{N_f}K^H_f\in R$. Then
\begin{equation}\label{eq:saturated-fill-ideal}
 I_{H,F}\sett (p_f:f\in F):\Delta_H^\infty
\end{equation}
is independent of the chosen exponents $N_f$ and $A_{H,F}\simeq R/I_{H,F}$.
\end{corollary}
\begin{proof}
$R = \K_{G}[x_{F}]$ is a domain so the inverse image of the ideal $(K_{f}^{H} : f\in F)\subseteq R_{\Delta_{H}}$ under the map $R\to R_{\Delta_{H}}$ is given by  \begin{align*}
(K_f^{H}:f\in F)R_{\Delta_H}\cap R &= (p_f:f\in F)R_{\Delta_H}\cap R &[\because \Delta_{H} \in (R_{\Delta_{H}})^{\times}]\\
&= \set{g\in R \st g/1 \in \lp p_{f}:f\in F\rp R_{\Delta_{H}}} \\
&= \set{g\in R \st \exists m\ge 0\text{ s.t. }\Delta_{H}^{m}g \in (p_{f}:f\in F)R}\\
&= (p_{f}:f\in F):\Delta_{H}^{\infty} = I_{H,F}.
\end{align*}
Every allowed choice of the exponents $N_f$ generates the same ideal after localization, so this contraction also proves the claimed independence of the numbers. Moreover, the kernel of the composition $R \to R_{\Delta_{H}} \to A_{H,F}$ is precisely $I_{H,F}$ by the above. Hence $R/I_{H,F}$ injects into $A_{H,F}$ as finite dimensional $\K_{G}-$algebras, by the first isomorphism theorem, and is thus finite dimensional. But $A_{H,F} = R_{\Delta_{H}}/(K_f^{H}:f\in F)R_{\Delta_H} = R_{\Delta_{H}}/(p_f:f\in F)R_{\Delta_H} \cong (R/I_{H,F})_{\Delta_{H}}$. Multiplication by $\Delta_{H}$ on $R/{I_{H,F}}$ is injective because $I_{H,F}:\Delta_{H} = I_{H,F}$. Since $R/{I_{H,F}}$ is finite dimensional (as it injects into the finite dimensional $A_{H,F}$), the endomorphism of the $\Delta_{H}-$multiplication map on $R/{I_{H,F}}$ must be an isomorphism since it is injective. This proves that $\Delta_{H} \in (R/{I_{H,F}})^{\times}$ whence $A_{H,F} \cong (R/{I_{H,F}})_{\Delta_{H}} \cong R/{I_{H,F}}$.
\end{proof}

The above reduction has a clean implication that the $\MLdeg$ of a chordal graph is one. Indeed taking $F = \varnothing$ in the above (so $G=H$) gives $\Delta_{H} \in \mathbb K_{G}^{\times}$ so that $A_{H,F} \cong \mathbb K_{G}$ whence $\MLdeg(G) = \dim_{\mathbb K_{G}}\mathbb K_{G} = 1$. However this is not a new proof because this is essentially an implication of the well-known \Cref{prop:chordal-formula}.

\subsection{Specializations for one-fill}\label{sec:specialization-one-fill}
Here we mention some calculations and lemmas that specialize to the case $\abs F=1$. The fill edge in $F$ would be $F=\set{uv}$ and the reader should keep this in mind while reading this section. Let $\mathbb F$ be a field of characteristic zero and let $B=\set{u,v}\sqcup W$. Suppose
\[
 \Sigma_B(x)=
 \begin{bmatrix}
 a&x&r^\top\\
 x&b&s^\top\\
 r&s&T
 \end{bmatrix}_{u,v,W},
\]
where $T$ is symmetric and invertible. Take
\[
 \alpha=a-r^\top T^{-1}r,
 \quad
 \beta=b-s^\top T^{-1}s,
 \quad
 c=r^\top T^{-1}s,
 \quad
 D_B(x)=\alpha\beta-(x-c)^2.
\]
These are the terms which occur in the Schur complement of $\Sigma_{B}(x)$ with respect to $T$. Indeed, 
\begin{equation}\label{eq:schur-terms}
S_{T}(\Sigma_{B}(x)) = \begin{bmatrix}a&x\\x&b\end{bmatrix} -\begin{bmatrix}r^\top\\s^\top\end{bmatrix} T^{-1}\begin{bmatrix}r&s\end{bmatrix} =\begin{bmatrix}\alpha&x-c\\x-c&\beta\end{bmatrix}.
\end{equation}
In particular, the diagonal entries of this symmetric matrix are entries of $\mathbb F$ and the anti-diagonal elements are monic linear polynomials.

\begin{lemma}[Local Schur-complement term]\label{lem:local-schur}
If $D_B(x)\neq0$ then
\begin{equation}\label{eq:local-schur-entry}
\det\Sigma_B(x)=\det(T)D_B(x), \qquad (\Sigma_B(x)^{-1})_{uv}=\frac{c-x}{D_B(x)}.
\end{equation}
\end{lemma}

\begin{proof}
Using \eqref{eq:schur-terms}, the determinant formula for a Schur complement gives the first identity in \eqref{eq:local-schur-entry}. The upper-left block of $\Sigma_B(x)^{-1}$ is the inverse of this Schur complement, namely
\[
 \frac{1}{D_B(x)}
 \begin{bmatrix}\beta&c-x\\c-x&\alpha\end{bmatrix},
\]
which gives the second identity. For the reader, the latter is usually referred to as the inverse formula for Schur complements.
\end{proof}

The $B$ above will be later treated as a `separator bag containing $uv$' where we will sum up rational expressions, over bags $B$, with denominator $D_{B}$ and linear numerators of the form $\lambda_{B}(c_{B}-x)$. Let $\cB$ be a nonempty finite set containing sets $B$ of the above form. One might expect that for generic data, these denominators are pairwise coprime so that the sum of such rational expressions should have the following numerator and denominator.
\begin{align}
D(x)&=\prod_{B\in\cB}D_B(x),\label{eq:schur-D}\\
N(x)&=\sum_{B\in\cB}\lambda_B(c_B-x) \prod_{B'\in\cB\smallsetminus\{B\}}D_{B'}(x). \label{eq:schur-N}
\end{align}
Under genericity and other assumptions, one can compute the degree of the aforementioned rational expression.

\begin{lemma}[Degree and coprimality of a Schur sum]
\label{lem:schur-sum-degree}
Let $m=\abs{\cB}$ and assume that the $D_B$ are pairwise coprime, $\alpha_B\beta_B\neq0$, $\lambda_B\in\mathbb F^\times$, and $\sum_{B\in\cB}\lambda_B=1$. Then
\[ \sum_{B\in\cB}\lambda_B\frac{c_B-x}{D_B(x)}= \frac{N(x)}{D(x)}, \qquad \deg N=2m-1, \qquad \gcd(N,D)=1.\]
Consequently
\begin{equation}\label{eq:schur-quotient-length}
 \mathbb F[x,D^{-1}]/(N/D)\simeq\mathbb F[x]/(N)
 \quad\text{and}\quad
 \dim_{\mathbb F}\mathbb F[x,D^{-1}]/(N/D)=2m-1.
\end{equation}
\end{lemma}

\begin{proof}
Since $D_{B}$'s are pairwise coprime, clearing denominators for the sum of the rational expressions gives $N/D$. Each $D_{B}$ is quadratic and $-D_{B}$ is monic. Therefore, the top coefficient (namely that of $x^{2m-1}$) of each summand in $N(x)$ has coefficient $\lambda_{B}\cdot (-1)\cdot (-1)^{m-1}$ where the first $-1$ comes from the term $c_{B}-x$ and the $(-1)^{m-1}$ comes from the term $\prod_{B'\in \cB\smallsetminus \set B}D_{B'}(x)$. Therefore the coefficient of $x^{2m-1}$ in $N(x)$ is the sum of $(-1)^{m}\lambda_{B}$'s which is $(-1)^{m}$. Therefore $\deg N = 2m-1$.

To prove coprimality, work in an algebraic closure of $\mathbb F$ and let
$\rho$ be a root of $D_B$. Pairwise coprimality gives
$D_{B'}(\rho)\neq0$ for every $B'\neq B$. Moreover
$\rho\neq c_B$, because $D_B(c_B)=\alpha_B\beta_B\neq0$. At $\rho$, all
summands of $N$ except the $B$-summand vanish, and hence
\[
 N(\rho)=\lambda_B(c_B-\rho) \prod_{B'\neq B}D_{B'}(\rho)\neq0.
\]
No root of $D$ is a root of $N$, so $\gcd(N,D)=1$. Bezout's identity then makes $D$ a unit in $\mathbb F[x]/(N)$. Its vector-space dimension is $\deg N$, which proves \eqref{eq:schur-quotient-length}.
\end{proof}

\section{One-fill chordal graphs}\label{sec:one-fill}

\begin{definition}[One-fill chordal graph]
A graph $G$ is \it{one-fill chordal} if a nonedge $uv$ of $G$ satisfies that $H=G+uv$ is chordal. The edge $uv$ of $H$ is called a \it{fill edge}.
\end{definition}

Throughout this section, let $H$ be chordal, let $uv\in E(H)$, take $G=H\smallsetminus uv$, and fix a clique tree $T$ of $H$. Write
\[ \mathbb K=\mathbb K_G, \qquad x=\Sigma_{uv}, \qquad \cB_{uv}=\cC_{uv}(H)\sqcup\cS_{uv}(H,T).\]

Recall from the last section the specialization for $F=\set{uv}$ that $y(x)$ is the partial covariance matrix on $H$ with entries $s_{ij}$ on $\cO_{G}$ and $x$ at position $(u,v)$. Further, \[\Delta_{H}(x) = \prod_{C\in \cC(H)}\det y_{C}(x)\prod_{e\in E(T)}\det y_{B_{e}}(x), \qquad R = \K[x].\] The chordal formula would then give, for the partially observed covariances $y(x)$, that the unique concentration matrix on $H$, whose inverse agrees with $y(x)$, is 
\begin{equation}\label{eq:chordal-solution-x}
K^{H}(x) = \sum_{C\in \cC(H)}(y_{C}(x)^{-1})^{0} - \sum_{e\in E(T)}(y_{B_{e}}(x)^{-1})^{0}.\end{equation}
The earlier reduction from the likelihood equations of $G$ to $H$ would then give the following.

\begin{lemma}[One-fill reduction]\label{lem:one-fill-reduction}
We have 
\begin{equation}\label{eq:one-fill-algebra-isomorphism}\MLdeg(G)= \dim_{\mathbb K} \left(R_{\Delta_H}/(K^H_{uv}(x))\right).
\end{equation}
\end{lemma}

\begin{proof}
Apply \Cref{lem:algebraic-fill-reduction} with $F=\set{uv}$.
\end{proof}

Using \eqref{eq:saturated-fill-ideal} in one variable, the above says that if $N_0(x)$ is any numerator obtained after clearing powers of
$\Delta_H$, then
\[
 \MLdeg(G)=
 \dim_{\mathbb K}
 \frac{\mathbb K[x]}{(N_0):\Delta_H^\infty}.
\]

Now we massage \eqref{eq:chordal-solution-x} at the entry $(u,v)$ since we want to force that entry to zero. By observation \eqref{obs:padding}, the unique terms (specifically the subscripts of $y$) appearing in the $(u,v)-$coordinate of this expression are given by the set $\cB_{uv} = \cC_{uv}(H) \sqcup \cS_{uv}(H,T)$. Let \[
\lambda_{B} = \lambda_{B}^{(uv)} \sett 
\begin{cases}
1 & \text{ if } B \in \cC_{uv}(H)\\
-\abs{\set{e\in E(T) : B_{e}=B}} & \text{ if } B \in \cS_{uv}(H)
\end{cases}.\]
Then \begin{equation}\label{eq:reduced-K}K_{uv}^{H}(x) = \sum_{B\in \cB_{uv}(H,T)} \lambda_{B}\cdot (y_{B}(x)^{-1})_{uv}.\end{equation} We will be done with the calculation of the ML degree of $G$ once we show that the above expression has the relevant structure and satisfies the assumptions as discussed in the results in \Cref{sec:specialization-one-fill}.

For every $B\in \cB_{uv}$ one can write $B = \set{u,v}\sqcup W_{B}$ and \begin{equation}\label{eq:y-expression}y_{B}(x) = \begin{bmatrix}
a_{B}&x&r_{B}^{\top}\\
x&b_{B}&s_{B}^{\top}\\
r_{B}&s_{B}&T_{B}
\end{bmatrix}
\end{equation}
where all entries except $x$ lie in $\K$ and $T_{B}$ is symmetric. As before, let 
\[
\alpha_{B}=a_{B}-r_{B}^\top T_{B}^{-1}r_{B}, \quad \beta_{B} = b_{B} - s_{B}^\top T_{B}^{-1}s_{B}, \quad c_{B}=r_{B}^\top T_{B}^{-1}s_{B}, \quad D_B(x)=\alpha_{B}\beta_{B}-(x-c_{B})^2.
\]
Recall that these are the terms appearing in the Schur complement of $y_{B}$ with respect to $T_{B}$.

\begin{lemma}\label{lem:bag-expansion}
Every $T_{B}$ is nonsingular over $\K$ and \begin{equation}\label{eq:bag-term}
K^H_{uv}(x)=\sum_{B\in\cB_{uv}} \lambda_B\frac{c_B-x}{D_B(x)}.
\end{equation}
\end{lemma}
\begin{proof}
With the specialization of $(s_{ij})$ which makes $T_{B}$ the identity matrix, it is clear that $\det T_{B}$ is not identically zero. Therefore it is invertible over $\K$. Applying the second identity of \eqref{eq:local-schur-entry} in \Cref{lem:local-schur} to every summand of \eqref{eq:reduced-K} gives \eqref{eq:bag-term}.
\end{proof}

The final step is to apply \Cref{lem:schur-sum-degree} and to do that, one has to verify the assumptions that $\alpha_{B}\beta_{B}\ne 0\fa B\in \cB_{uv}$, all $\set{D_{B}}_{B\in \cB_{uv}}$ are pairwise coprime and that the $(\lambda_{B})_{B\in \cB_{uv}}$ sum to one and are all nonzero. The last condition is true by \Cref{lem:properties}\ref{sum-net-coefficients}.

\begin{lemma}
\label{lem:generic-denominators}
Over $\mathbb K$, the numbers $\alpha_B\beta_B$ are nonzero and the quadratics $D_B(x)$, for $B\in\cB_{uv}$, are pairwise coprime.
\end{lemma}
\begin{proof}
On the open set where all $\det T_{B}\ne 0$, the $\alpha_{B}\beta_{B}$ and the pairwise resultants $\operatorname{Res}_x(D_B,D_{B'})$ are rational expressions of the observed entries. After clearing powers of the $\det T_{B}$'s, their simultaneous non-vanishing is a Zariski-open condition and thus it is enough to show that this open set is non-empty. Note that we have used here that $\gcd(D_{B},D_{B'})=1\iff \operatorname{Res}_x(D_B,D_{B'})\ne 0$. It is therefore enough to show that these conditions hold for some specialization of $s_{ij}$. 

Conveniently we choose them such that, referring to \eqref{eq:y-expression}, every $T_{B}$ is identity and every $s_{B}$ is the zero vector which would make \[\alpha_{B} = s_{uu} - \|r_{B}\|^{2},\qquad \beta_{B} = s_{vv}, \qquad c_{B} = 0,\qquad D_{B}(x) = (s_{uu}-\|r_{B}\|^{2})s_{vv}-x^{2}.\] It is enough to take $\|r_{B}\|^{2} \ne \|r_{B'}\|^{2}$ (and not equal to $s_{uu}$) for $B\ne B'\in \cB_{uv}$ while ensuring $s_{vv}\ne 0$, in which case all the quadratics $D_{B}$ would have pairwise distinct unequal roots. Putting this together, we take the following specialization.

Order the vertices $V(H)\smallsetminus \set{u,v} = \set{w_{0}, \cdots, w_{t-1}}$ and take \[s_{uu} = 2^{t}, \qquad s_{vv} = 1, \qquad s_{u,w_{i}}=2^{i/2} \fa uw_{i}\in E(G), \qquad s_{w_{i}w_{i}}=1\fa 0\le i<t\] and all other observed off-diagonal entries as zero. Then $\|r_{B}\|^{2} = \sum_{w_{i\in B\smallsetminus\set{u,v}}}2^{i} < 2^{t}$. These numbers, for distinct $B$, are all distinct because of the distinctness of the indices contained in $B$ and distinct sets of indices determine distinct binary numbers. Also $s_{uu} \in \C^{\times} \smallsetminus \set{\|r_{B}\|^{2} \st B\in \cB_{uv}}$. This makes all $\alpha_{B}\ne 0, \beta_{B}\ne 0$ and all $\set{D_{B}}_{B\in \cB_{uv}}$ have no pairwise common root, that is, pairwise coprime.
\end{proof}

\begin{proof}[Proof of \Cref{thm:one-fill}]
Let \[D(x) = \prod_{B\in \cB_{uv}}D_{B}(x).\] Observe by the first identity in \eqref{eq:local-schur-entry} of \Cref{lem:local-schur} and invertibility of $\det T_{B}$ by \Cref{lem:bag-expansion} that $\Delta_{H}$ and $D$ have the same factors $\set{D_{B}(x)}$, but $\Delta_{H}$ is a product of them with certain exponents, and $D$ is simply a product, up to units in $\K_{G}$. Consequently, $\K_{G}[x,\Delta_{H}^{-1}] = \K_{G}[x,D^{-1}]$. By this observation, \Cref{lem:one-fill-reduction}, we have $\MLdeg(G) = \dim_{\K}(\K[x,D^{-1}]/(K^{H}_{uv}(x)))$. Since the conditions of \Cref{lem:schur-sum-degree} are satisfied by \Cref{lem:bag-expansion,lem:generic-denominators} whence \eqref{eq:bag-term} gives $K^{H}_{uv} = N(x)/D(x)$ where $\gcd(N,D)=1, \deg N = 2\abs{\cB_{uv}}-1$ and $\dim_{\K}\K[x,D^{-1}]/(N/D) = \deg N$. Therefore \[\MLdeg (G) = 2\abs{\cB_{uv}}-1 = 2\lp\abs{\cC_{uv}(H)} + \abs{\cS_{uv}(H,T)}\rp-1.\]
\end{proof}

\begin{remark}[When deletion remains chordal]
If $H\smallsetminus uv$ is chordal, then $uv$ lies in only one maximal clique of $H$. Indeed, otherwise the subtree $T_{uv}$ has an edge $CC'$. Choose $a\in C\smallsetminus C'$ and $b\in C'\smallsetminus C$. These sets are nonempty because $C,C'$ are distinct maximal cliques. Moreover $ab\notin E(H)$: if a maximal clique $D$ contained $a,b$, then $D$ lies on one side of the edge $CC'$; applying the path property to $a$ or to $b$, according to the side, would force $a\in C'$ or $b\in C$. Hence $u,a,v,b,u$ is an induced four-cycle in $H\smallsetminus uv$, a contradiction. Thus $\abs{\cC_{uv}(H)}=1$ and $\cS_{uv}(H,T)=\varnothing$, and the formula correctly gives ML degree $1$.
\end{remark}

\subsection{Application for Deleted-spine clique books}

\begin{corollary}\label{cor:book}
Let $n\geq2$ and let $G=\operatorname{Book}(A_{1},\ldots,A_{n})$. Then $\MLdeg(G)=2n+1$. For $n=1$, the graph is chordal and has ML degree $1$.
\end{corollary}

\begin{proof}
$H=G+uv$ which is chordal. Indeed, a cycle that meets two different pages must use both $u$ and $v$. They cannot be consecutive on such a cycle, since the remaining $u-v$ path would lie in a single component of $H\smallsetminus\set{u,v}$, hence in a single page. Thus $uv$ is a chord. A cycle contained in $\set{u,v}\cup A_{i}$ lies in a complete graph. 

Observe that the maximal cliques of $H$ are $C_{i}=\set{u,v}\cup A_{i}$ for $i=1,\cdots,n$. Any tree with vertex set $\set{C_{1},\cdots,C_{n}}$ is a clique tree, and every one of its edge separators is $\set{u,v}$. For $n\geq2$, we have $\abs{\cC_{uv}(H)}=n, \cS_{uv}(H,T)=\set{\set{u,v}}$, whence \ref{eq:off} gives $\MLdeg(G) = 2n+1$. When $n=1$, observe $G = K_{A_{1}\cup\set{u,v}}\smallsetminus uv = (K_{A_{1}} * u) \cup_{K_{A_{1}}} (K_{A_{1}} * v)$. By \Cref{lem:mult-clique-sum,thm:suspension}, $\MLdeg(G) = \MLdeg(K_{A_{1}})^{2}=1$.
\end{proof}

\begin{example}[$K_{2,n}$]
If each $A_{i}$ is a singleton and $n\geq2$, then the \dsb~ is $K_{2,n}$. \Cref{cor:book} gives $\MLdeg(K_{2,n})=2n+1$. For $n=1$, $K_{2,1}=P_3$ is chordal with $\MLdeg=1$.
\end{example}

\section{Several noninteracting fill edges}\label{sec:several-fill}

We now move on to the case where we analyze the case of several fill edges to reach a chordal completion. The corresponding chordal-completion equations are multivariate -- one variable for each fill edge. Under the independence assumption discussed later (which we call bag noninteracting), these equations split into independent subsystems of univariate polynomials. 

Let $H$ be chordal, and fix a clique tree $T$ of $H$. For $e = CC' \in E(T)$, recall $B_{e}= C\cap C'$. The set of distinct clique and separator bags is $\cB(H,T) = \cC(H) \sqcup \cS(H,T)$ where $\cC(H)$ is the set of maximal cliques and $\cS(H,T) = \set{B_{e} \st e\in E(T)}$. If $F \subseteq E(H)$ and $B \in \cB(H,T)$, let $F_{B} \sett \set{ab\in F\st \set{a,b}\subseteq B} = F \cap {B \choose 2}$.

\begin{definition}[Fill-edge interaction graph]\label{def:interaction-graph}
The \it{interaction graph} $\Gamma_T(F)$ has vertex set $F$. Two distinct fill edges $f,g\in F$ are adjacent if $f,g\in F_B$ for some $B\in\cB(H,T)$.
\end{definition}

Recall that for $f=uv\in F$, we write \[
\cC_{f}(H) = \set{C\in \cC(H) \st \set{u,v} \subseteq C}, \qquad \cS_{f}(H,T) = \set{B_{e} \st e\in E(T),\set{u,v}\subseteq B_{e}}
\]
where the latter is an honest set without repeated elements. The multiplicity of a particular $B\in \cS_{f}(H,T)$ among the edges $e\in E(T)$ is precisely $-\lambda_{B}^{(f)}$, that is, $-\lambda_{B}^{(f)} = \abs{\set{e\in E(T) \st B=B_{e}}}$. The set of relevant bags for $f$ is $\cB_{f}(H,T) = \cC_{f}(H) \sqcup \cS_{f}(H,T)$. Since we now have multiple fill edges in $F$, denote $m_{f} \sett \abs{\cB_{f}(H,T)}$. Note that by \Cref{thm:one-fill}, $2m_{f}-1 = \MLdeg(H\smallsetminus f)$ whence $m_{f}$ is independent of the chosen clique tree $T$. Write $x_{uv} = x_{f}$ for the unspecified covariance entry $\Sigma_{uv}$ for $f=uv\in F$. Take \[\mu_{B} \sett \begin{cases}
1 & \text{ if } B\in \cC(H)\\
-\abs{\set{e\in E(T)\st B=B_{e}}} & \text{ if } B \in \cS(H,T)
\end{cases}.\] Recall that in our analysis earlier, we used the following localization.
\[
\Delta_{H}(x_{F}) = \prod_{B\in \cB(H,T)} \det \Sigma_{B}(x_{F})^{\abs{\mu_{B}}}, \qquad R = \K_{G}[x_{F}, \Delta_{H}^{-1}].
\]
The corresponding likelihood ideal is $J_{F} = (K_{f}^{H}(x_{F}):f\in F)R$ whence by \Cref{lem:algebraic-fill-reduction}, $\MLdeg(G) = \dim_{\K_{G}}(R/J_{F})$.

We now try to study the equations under `factorizations'. A convenient break-down of the overall system into smaller systems is given by the non-interaction between equations whose variables correspond to different connected components of the interaction graph on $F$. Let $F=F_{1}\sqcup F_{2}\sqcup\cdots\sqcup F_{s}$ be the decomposition of the interaction graph $\Gamma_{T}(F)$ into connected components where $F_{i}$ contains the vertices in the $i^{\text{th}}$ connected component. Let \[\Delta_{j} = \prod_{\substack{B\in \cB(H,T) \\ \varnothing \ne F_{B}\subseteq F_{j}}} \det\Sigma_{B}(x_{F})^{\abs{\mu_{B}}}.\] Now we are interested in the component-wise algebras \[R_{j} \sett \K_{G}[x_{F_{j}},\Delta_{j}^{-1}], \qquad J_{j}\sett (K_{f}^{H}:f\in F_{j})R_{j}, \qquad A_{j} \sett R_{j}/J_{j}.\] By factorization of equations we precisely mean the following.

\begin{lemma}\label{lem:component-factorization}
Every $A_{j}$ is nonzero and finite dimensional $\K_{G}-$algebra. Further, \[\MLdeg(H\smallsetminus F) = \prod_{j=1}^{s}\dim_{\K_{G}}A_{j}.\]
\end{lemma}

\begin{proof}
If $f\in F_{j}\cap F_{B}$ for some $j\in [s], B\in \cB(H,T)$, then $F_{B}\subseteq F_{j}$. Indeed if $f\in F_{B}$ then for every $g\in F_{B}\smallsetminus \set{f}$, the vertices $f,g$ are adjacent in $\Gamma_{T}(F)$ by definition, so $g\in F_{j}$. It follows that $(\Sigma_{B}^{-1})_{f} \in \K_{G}[x_{F_{B}}, \det \Sigma_{B}^{-1}] \subseteq R_{j}$ whenever $f\in F_{j}$. By the clique separator formula \eqref{eq:chordal-formula} in \Cref{prop:chordal-formula} and observation \eqref{obs:padding}, we have 
\begin{equation}\label{eq:component-dependence}
K^{H}_{f} = \sum_{\substack{B\in\cB(H,T)\\f\in F_B}} \lambda_{B}^{(f)}(\Sigma_B^{-1})_f \in R_{j} \text{ whenever }f\in F_{j}.
\end{equation}
The equations $K_{f}^{H}=0$ split similarly. If $F_{B}\neq\varnothing$, then all edges in $F_B$ belong to a unique component $F_j$. Every determinant associated with a bag $B$ satisfying $F_B=\varnothing$ is a nonzero element of $\K_{G}$. Hence, for some $u\in\K_{G}^\times$, we have $\Delta_{H}=u\prod_{j=1}^s\Delta_j$. Due to the separation of the global determinant for $F$ into component determinants, we get the decomposition $R\cong R_{1}\otimes_{\K_{G}}\cdots \otimes_{\K_{G}}R_{s}$ of the global localized algebra into componentwise localized algebras. Under this identification, we have  $J_{F} =\sum_{j\in[s]}J_{j}R$ due to the separation of the likelihood equations from \eqref{eq:component-dependence}. $J_{F}$ is precisely the kernel of the surjective homomorphism \[\bigotimes_{j=1}^{s}R_{j}\twoheadrightarrow \bigotimes_{j = 1}^{s} R_{j}/J_{j}, \qquad a_{1}\otimes \cdots\otimes a_{s} \mapsto (a_{1}+J_{1})\otimes \cdots\otimes (a_{s}+J_{s}).\] By the first isomorphism theorem, 
\[ 
R/J_F \cong \bigotimes_{j=1}^s R_j/J_j.
\]
By the discussion above with application of \Cref{lem:algebraic-fill-reduction}, $R/J_{F}$ is finite dimensional. Since each $A_{j} = R_{j}/J_{j}$ embeds into $R/J_{F}$, we must have $\dim_{\K_{G}}(R_{j}/J_{j}) < \infty$. Taking dimensions of the above isomorphism gives $\MLdeg(H\smallsetminus F) = \prod \dim_{\K_{G}}A_{j}$.
\end{proof}

The independence assumption we were talking about earlier is precisely the condition (again due to convenience) that every component $F_{j}$ above is a singleton. Such an assumption makes each $J_{j}$ be generated by only one equation and then one can simply invoke the earlier analysis which would give us the computation for the dimension of each likelihood algebra $A_{j}$.

\begin{definition}[Bag-noninteracting fill set]\label{def:bag-noninteracting}
The set $F$ is \it{bag-noninteracting} if $|F_B|\leq 1$ for every $B\in\cB(H,T)$.
\end{definition}

Note that bag non-interaction is not dependent on the choice of the clique tree $T$ of $H$. Indeed suppose $F$ satisfies $\abs{F\cap {B\choose 2}} \le 1$ for every $B\in \cB(H,T)$. In particular this is true for every maximal clique $B=C\in \cC(H)$. This implies that $\abs{F\cap {C\cap C'\choose 2}} \le 1$ for every $C,C'\in \cC(H)$. 

\begin{theorem}[Product formula for bag-noninteracting fills]
\label{thm:bag-noninteracting}
Let $H$ be chordal, let $T$ be a clique tree of $H$, and let $F\subseteq E(H)$ be bag-noninteracting. Then
\[
 \MLdeg(H\smallsetminus F)=\prod_{f\in F}(2m_{f}-1).
\]
\end{theorem}

\begin{proof}
Bag-noninteraction says exactly that $\Gamma_{T}(F)$ has no edge, so its components are the singletons $\set{f}$, for $f\in F$. Fix $f=uv$. A bag contributing to the sum \eqref{eq:component-dependence} in $K^{H}_{f}$ contains $f$ and hence, by bag-noninteraction, contains no other edge of $F$. Thus all entries of that bag except $x_{f}=\Sigma_{uv}$ lie in $\K_{G}$. \Cref{lem:schur-sum-degree,lem:bag-expansion} give
\[
 K^{H}_{f}(x_{f})= \sum_{B\in\cB_f(H,T)} \lambda_{B}^{(f)}\lp-\frac{x_{f}-c_{B}}{D_{B}(x_{f})}\rp, \qquad D_{B}(x_{f}) = \alpha_{B} \beta_{B}-(x_{f}-c_{B})^{2}.
\]
The coefficients satisfy $\sum_{B\in\cB_{f}(H,T)}\lambda_{B}^{(f)}=1$. Moreover, the generic nonvanishing and coprimality conditions in \Cref{lem:generic-denominators} involve only entries of these relevant bags. Those entries, other than $x_{f}$, are all elements of $\K_{G}$. Hence the proof of \Cref{lem:generic-denominators} gives nonzero polynomials in the generic data on $G$. Since there are only finitely many $f$, the required conditions hold simultaneously for generic data.

By \Cref{lem:schur-sum-degree}, the reduced fraction for $K^{H}_{f}$ has a numerator which is coprime to its denominator and this numerator has degree $2m_{f}-1$. Consequently $\dim_{\K_{G}}A_{f}=2m_{f}-1$. Applying \Cref{lem:component-factorization} to the singleton components gives the claimed $\MLdeg$.
\end{proof}

\section{Towards clique blow-up invariance of ML degree}\label{sec:blow-up}

Clique blow-up does not preserve ML degree in general. We first give a six-vertex counterexample to \Cref{qs:clique-blow-up} and then prove four positive results for \Cref{qs:special-vertex}.

\begin{example}[Failure of blow-up invariance]
\label{ex:double-blow-up}
Start with the cycle $C_4=1-2-3-4-1$ and successively replace the adjacent vertices $1$ and $2$ by two-cliques. After relabelling, the final graph $D$ is as in \Cref{fig:counter-example-blow-up}. First blowup preserves ML degree as five, but the second blowup increases it to nine.
\begin{figure}[h]\centering\begin{tikzcd}
	&& {\textcircled{\raisebox{-.9pt}{\color{red}1}}} & \\
	{\textcircled{\raisebox{-.9pt}6}} &&& {\textcircled{\raisebox{-.9pt}{\color{red}2}}} \\
	\\
	{\textcircled{\raisebox{-.9pt}5}} &&& {\textcircled{\raisebox{-.9pt}{\color{blue}3}}} \\
	&& {\textcircled{\raisebox{-.9pt}{\color{blue}4}}}
	\arrow[no head, from=1-3, to=2-4]
	\arrow[no head, from=1-3, to=4-4]
	\arrow[no head, from=1-3, to=5-3]
	\arrow[no head, from=2-1, to=1-3]
	\arrow[no head, from=2-1, to=2-4]
	\arrow[no head, from=2-1, to=4-1]
	\arrow[no head, from=2-4, to=4-4]
	\arrow[no head, from=2-4, to=5-3]
	\arrow[no head, from=4-1, to=4-4]
	\arrow[no head, from=4-1, to=5-3]
	\arrow[no head, from=4-4, to=5-3]
\end{tikzcd}\caption{A counterexample that blow-up does not preserve $\MLdeg$}
\label{fig:counter-example-blow-up}
\end{figure}
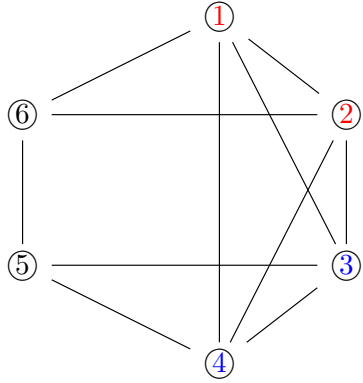
\begin{proposition}\label{prop:double-blow-up-degree}
Let $D$ be the graph in \Cref{fig:counter-example-blow-up}. Then $\MLdeg(D)=9$.
\end{proposition}
\begin{proof}[Proof sketch]
See \Cref{sec:code} for a detailed proof. Consider the chordal graph $H=D+\set{15,25}$. Solving for these fill edges automatically completes the matrix variables for edges $36,46$. Let $x=\Sigma_{15}, y=\Sigma_{25}$. The unique clique tree on $H$ has vertices $C_{1}=\set{1,2,3,4,5}, C_{2}=\set{1,2,5,6}$ and edge label $S=\set{1,2,5}$. Chordal formula says $\partial_{z}\log \det \Sigma = \partial_{z}\log(q_{1}q_{2}/q_{0})$ where $q_{0}=\det \Sigma_{S}, q_{1}=\det \Sigma_{C_{1}}, q_{2}=\det \Sigma_{C_{2}}$ and $z\in \set{x,y}$. $\partial_{z}\log \det\Sigma$ is then a rational form with numerator a quintic polynomial for both $z=x,y$ and denominator $q_{0}q_{1}q_{2}$. The numerators form a system of two quintic equations whence they have $25$ solutions by Bezout. But one must also discount the roots of the denominator because the $\MLdeg$ is the number of roots away from the poles. One can show that $q_{0},q_{1}$ intersect at four distinct points. Same is true for $q_{1},q_{2}$. But $q_{0},q_{2}$ intersect at two distinct points with multiplicity two each. There is also no triple intersection. These intersection numbers are valid for generic data and can be proven simply by one witness computed in \texttt{M2}, because these statements are Zariski open conditions in the data space. Now at a point where only one $q_{i}$ vanishes, the numerator is nonzero, for generic data. The only other case is to remove roots of the numerator where two of $q_{i}$'s vanish. For the pair $(q_{0},q_{1})$ or $(q_{1},q_{2})$, since there are four distinct points, the numerator also vanishes at these four distinct points with multiplicity one each. For the pair $(q_{0},q_{2})$, there are two distinct roots with multiplicity two at each root. For both these roots, Taylor expansion yields that numerators of both $d_{z}\det\Sigma$ for $z=x,y$ have the smallest degree terms as quadratics, so they look quadratic locally. Further these leading terms are coprime. So the local intersection multiplicity for the numerator is four locally at each tangency point. Therefore the number of roots of the numerators, which occur at the poles, is $4\cdot 1 + 4\cdot 1 + 2 \cdot 4 = 16$ whence $\MLdeg(D)=25-16=9$.
\end{proof}

\end{example}

Towards \Cref{qs:special-vertex}, we prove blow-up invariance for the four types of vertices mentioned in Table \eqref{eq:table-of-blowups}. First let's recall \Cref{def:clique-blow-up} of a clique blow-up. $\tilde G = \bl_{w}^{q}(G)$ is said to be the blow-up of $G$ at a vertex $w \in V(G)$ by the clique $K_{q}$ (which are new vertices, so $K_{q}\cap V(G) = \varnothing$) if $w$ is replaced with $K_{q}$ and every vertex of $\cN_{w,G}$ is connected to every vertex of $K_{q}$.

\subsection{Clique blow-up at a simplicial vertex}
$v\in V(G)$ is simplicial if $G[\cN_{v,G}]$ is a clique.
\begin{proposition}
\label{prop:simplicial-blow-up}
If $v$ is simplicial vertex of $G$, then $\MLdeg(\bl_v^q(G))=\MLdeg(G)=\MLdeg(G\smallsetminus v) \fa q\geq 1$.
\end{proposition}

\begin{proof}
Recall that \Cref{lem:mult-clique-sum} says $\MLdeg(G_{1}\cup_{K} G_{2}) = \MLdeg(G_{1})\MLdeg(G_{2})$. 

Let $N=\cN_{v,G}$. So $N$ is a clique. Note that $G = (G\smallsetminus v) \cup_{N} K_{N\cup \set v}$. Therefore $G$ is a clique sum of a complete graph and $G\smallsetminus v$ along the clique $N$ whence $\MLdeg(G\smallsetminus v) = \MLdeg(G)$. 

Let $K_{q} = \set{v_1,\ldots,v_q}$ be the clique replacing $v$. The vertices $N\cup K_{q}$ span a complete graph in $\bl_v^q(G)$ and $N$ is a clique in $G\smallsetminus v$. Therefore $\bl_v^q(G) = (G\smallsetminus v)\cup_{N} K_{N\cup K_{q}}$ whence $\MLdeg(\bl_v^q(G)) = \MLdeg(G\smallsetminus v)\MLdeg(K_{N\cup K_{q}}) = \MLdeg(G\smallsetminus v)$.
\end{proof}

\Cref{fig:simplicialpic} shows a conceptual proof-by-picture of the above argument.
\begin{figure}[H]\label{fig:simplicialpic}
\centering
\begin{tikzpicture}[
  every node/.style={font=\small},
  vertex/.style={circle, draw, fill=white, inner sep=1.5pt, minimum size=6pt},
  bluevertex/.style={circle, draw=blue!60!black, fill=blue!15, inner sep=1.5pt, minimum size=6pt},
  redvertex/.style={circle, draw=red!70!black, fill=red!15, inner sep=1.5pt, minimum size=6pt},
  graphbox/.style={draw, rounded corners, fill=gray!8, minimum width=4.2cm, minimum height=2.35cm},
  cliquebox/.style={draw=red!70!black, rounded corners, fill=red!5},
  nbox/.style={draw=blue!60!black, rounded corners, fill=blue!5},
  >=stealth
]


\node[graphbox] (Gminus) at (0,0) {};
\node at (0,-1.55) {${\color{blue}G\smallsetminus v}$};

\node[bluevertex] (a1) at (-0.75,0.35) {};
\node[bluevertex] (a2) at (0,0.75) {};
\node[bluevertex] (a3) at (0.75,0.35) {};

\draw[blue!60!black, thick] (a1)--(a2)--(a3)--(a1);

\node[vertex] (u1) at (-1.25,-0.45) {};
\node[vertex] (u2) at (1.25,-0.45) {};
\draw[gray] (u1)--(a1);
\draw[gray] (u2)--(a3);

\node[redvertex] (v) at (0,1.75) {$v$};
\draw[red!70!black, thick] (v)--(a1);
\draw[red!70!black, thick] (v)--(a2);
\draw[red!70!black, thick] (v)--(a3);

\begin{scope}[on background layer]
  \node[fit=(a1)(a2)(a3), nbox, inner sep=5pt] (Nvleft) {};
  \node[fit=(v)(a1)(a2)(a3), cliquebox, inner sep=7pt] (Kleft) {};
\end{scope}

\node[blue!60!black, below=-8pt of Nvleft] {$\cN_v$ is a clique};
\node[red!70!black, above=2pt of Kleft] {$K_{\cN_v\cup\{v\}}$};

\node at (0,-2.20) {
$\displaystyle
G=(G\smallsetminus v)\cup_{\cN_v}K_{\cN_{v}\cup\{v\}}
$
};


\node at (3.8,0.25) {$\leadsto$};
\node[align=center] at (3.8,-0.55) {replace $v$ by\\a clique $K_{q}$};


\begin{scope}[xshift=7.6cm]

\node[graphbox] (GminusR) at (0,0) {};
\node at (0,-1.55) {${\color{blue}G\smallsetminus v} \simeq \bl_{v}^{q}(G) \smallsetminus {\color{red}K_{q}}$};

\node[bluevertex] (b1) at (-0.75,0.35) {};
\node[bluevertex] (b2) at (0,0.75) {};
\node[bluevertex] (b3) at (0.75,0.35) {};

\draw[blue!60!black, thick] (b1)--(b2)--(b3)--(b1);

\node[vertex] (w1) at (-1.25,-0.45) {};
\node[vertex] (w2) at (1.25,-0.45) {};
\draw[gray] (w1)--(b1);
\draw[gray] (w2)--(b3);

\node[redvertex] (q1) at (-0.55,1.75) {$v_1$};
\node[redvertex] (q2) at (0.55,1.75) {$v_2$};
\node[redvertex] (q3) at (0,2.35) {$v_q$};

\draw[red!70!black, thick] (q1)--(q2)--(q3)--(q1);

\foreach \q in {q1,q2,q3} {
  \foreach \b in {b1,b2,b3} {
    \draw[red!70!black, thick] (\q)--(\b);
  }
}

\begin{scope}[on background layer]
  \node[fit=(b1)(b2)(b3), nbox, inner sep=5pt] (Nvright) {};
  \node[fit=(q1)(q2)(q3)(b1)(b2)(b3), cliquebox, inner sep=7pt] (Kright) {};
\end{scope}

\node[blue!60!black, below=-8pt of Nvright] {$\cN_{v}$ is the same clique};
\node[red!70!black, above=2pt of Kright] {$K_{\cN_v\cup K_{q}}$};

\node at (0,-2.20) {
$\displaystyle
\bl_v^q(G)
=
(G\smallsetminus v)\cup_{\cN_v}K_{\cN_v\cup K_{q}}
$
};

\end{scope}

\end{tikzpicture}

\caption{A simplicial vertex $v$ satisfies that $\cN_v$ is a clique. Replacing $v$ by a clique $K_{q}$ only replaces the complete summand $K_{\cN_v\cup\{v\}}$ by $K_{\cN_v\cup K_{q}}$, while the gluing clique $\cN_v$ is unchanged. }
\end{figure}
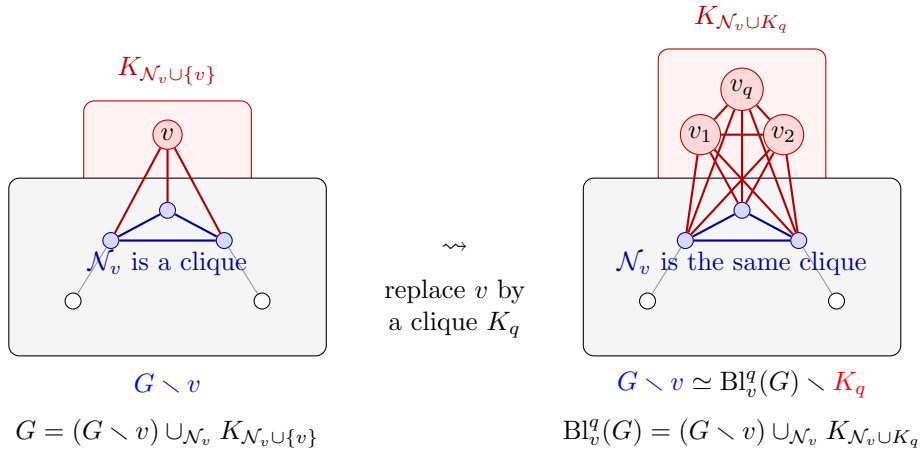


\begin{remark}
\Cref{lem:mult-clique-sum}, used above, is false if one were to replace `gluing along clique' with `gluing along a chordal graph'. Indeed consider the example, $H$ as the path $1-2-3$, $G_{1}$ as the cyclic graph on vertices $\set{a,1,2,3}$ with edges $a-1-2-3-a$, $G_{2}$ as the cyclic graph on $\set{b,1,2,3}$ with edges $b-1-2-3-b$. Then gluing $G_{1},G_{2}$ along the chordal graph (in fact, a tree) $H$ gives $K_{2,3}$ with bipartition $\set{1,3},\set{a,2,b}$. But $\MLdeg(K_{2,3}) = 7$ whereas $\MLdeg(G_{1})\MLdeg(G_{2}) = 25$.
\[\begin{tikzcd}[column sep=2.25em]
	& 1 & 1 &&& \textcolor{rgb,255:red,92;green,92;blue,214}{1} & \\
	a & 2 & 2 & b & \textcolor{rgb,255:red,214;green,92;blue,92}{a} & \textcolor{rgb,255:red,214;green,92;blue,92}{2} & \textcolor{rgb,255:red,214;green,92;blue,92}{b} \\
	& 3 & 3 &&& \textcolor{rgb,255:red,92;green,92;blue,214}{3}
	\arrow[color={rgb,255:red,117;green,117;blue,117}, between={0.2}{0.8}, Rightarrow, scaling nfold=3, squiggly, no head, from=1-2, to=1-3]
	\arrow[no head, from=1-2, to=2-1]
	\arrow[no head, from=1-2, to=2-2]
	\arrow[no head, from=1-3, to=2-4]
	\arrow[no head, from=1-6, to=2-6]
	\arrow[no head, from=1-6, to=2-7]
	\arrow[no head, from=2-1, to=3-2]
	\arrow[color={rgb,255:red,117;green,117;blue,117}, between={0.2}{0.8}, Rightarrow, scaling nfold=3, squiggly, no head, from=2-2, to=2-3]
	\arrow[no head, from=2-2, to=3-2]
	\arrow[no head, from=2-3, to=1-3]
	\arrow[squiggly, from=2-4, to=2-5]
	\arrow[no head, from=2-4, to=3-3]
	\arrow[no head, from=2-5, to=1-6]
	\arrow[no head, from=2-5, to=3-6]
	\arrow[no head, from=2-6, to=3-6]
	\arrow[color={rgb,255:red,117;green,117;blue,117}, between={0.2}{0.8}, Rightarrow, scaling nfold=3, squiggly, no head, from=3-2, to=3-3]
	\arrow[no head, from=3-3, to=2-3]
	\arrow[no head, from=3-6, to=2-7]
\end{tikzcd}\]
\end{remark}

\subsection{Clique blow-up at a universal vertex}
The next case is different. A universal vertex need not be simplicial, since its neighborhood need not be a clique. Nevertheless, clique blow-up at a universal vertex also preserves ML degree by \Cref{thm:suspension}. Here we discuss about blowup at a universal vertex, that is, a vertex which is adjacent to every other vertex.

\begin{proposition}\label{prop:universal-blow-up}
If $v$ is a universal vertex in $G$, then $\MLdeg(\bl_v^q(G))=\MLdeg(G) = \MLdeg(G\smallsetminus v)\fa q\ge 1$.
\end{proposition}

\begin{proof}
Since $v$ is universal, $G=(G\smallsetminus v)*K_1$, where $*$ denotes graph join. Replacing $v $ by a $q $-clique gives $\bl_v^q(G)=(G\smallsetminus v)*K_q$. \Cref{thm:suspension} says that adding $t$ universal vertices preserves the ML degree. So $\MLdeg((G\smallsetminus v)*K_t)=\MLdeg(G\smallsetminus v)$. Since $G\simeq (G\smallsetminus v)*K_1 $, the above equality for $t=1,q$ implies $\MLdeg(\bl_v^q(G))=\MLdeg(G) = \MLdeg(G\smallsetminus v).$
\end{proof}

\subsection{Clique blow-up away from the fill edge}
The above two results do not explain \dsb s. However, the next result explains, among other examples, why replacing the vertices on the $n$-side of $K_{2,n}$ by arbitrary cliques does not change its ML degree. Let $G$ be a one fill chordal graph with $H=G+uv$ chordal. Then a vertex $w$ of $G$ is said to be away from the fill edge if $w\notin \set{u,v}$.

\begin{proposition}\label{prop:onefill-blow-up}
Let $H$ be chordal, $uv\in E(H) $, and set $G=H\smallsetminus uv$. Let $w\in V(H)\smallsetminus\set{u,v}$. Then $\MLdeg(\bl_w^q(G)) = \MLdeg(G)\fa q\ge 1$.
\end{proposition}
\begin{proof}
Let $Q\simeq K_q$ replace $w$, and define a map on subsets of $V(H)$ by
\[
\phi(A) = \begin{cases} 
(A\smallsetminus\set{w})\cup Q,&w\in A,\\ 
A,& w\notin A.
\end{cases}
\]
Let $\tilde H = \bl_{w}^{q}(H)$. If $C$ is a maximal clique of $H$, then $\phi(C)$ is a maximal clique of $\tilde H$. Indeed, any strict clique extension of $\phi(C)$ would, after replacing $Q$ by $w$, give a strict clique extension of $C$ in $H$. Conversely, replacing the vertices of $Q$ in any clique of $\tilde H$ by $w$ gives a clique of $H$. Any maximal clique of $\tilde H$ that meets $Q$ contains all of $Q$, because the vertices of $Q$ are pairwise adjacent true twins. Consequently every maximal clique of $\tilde H$ is $\phi(C)$ for a unique maximal clique $C$ of $H$. Uniqueness follows because membership in $Q$ records membership of $w$, after which the original subset is recovered. Thus $\cC(\tilde H)=\set{\phi(C)\st C\in\cC(H)}$.

Fix a clique tree $T$ of $H$. Replace each node $C$ of $T$ by $\phi(C)$ to obtain the tree $\tilde T$. We show that $\tilde T$ is a clique tree of $\tilde H$. The running-intersection property holds for every old vertex because it held in $T$. For each $z\in Q$, the nodes containing $z$ are precisely those whose original labels contained $w$, and these nodes also form a connected subtree. Hence the resulting tree $\tilde T$ is a clique tree of $\tilde H$. Moreover, $\phi(C)\cap\phi(C')=\phi(C\cap C')$ for every edge $CC'$ of $T$. The map $\phi$ is injective on subsets, so it also preserves the number of distinct separator bags.

Because $w\notin\set{u,v}$, a clique or separator bag contains both $u$ and $v$ if and only if its image under $\phi$ does. Therefore $\abs{\cC_{uv}(\tilde H)} = \abs{\cC_{uv}(H)}, \abs{\cS_{uv}(\tilde H,\tilde T)} = \abs{\cS_{uv}(H,T)}$. Applying  \Cref{thm:one-fill} to $G=H\smallsetminus uv, \tilde G = \bl_{w}^{q}(G) = \tilde H \smallsetminus uv$ proves the claim.
\end{proof}

\subsection{Clique blow-up at a degree-two vertex}
We next move on to developing the blow-up invariance of a degree-two vertex. This is fundamentally different in the sense that this is not an implication of the \ref{eq:off} or results from \cite{SturmfelsUhler2010}. Suppose a vertex $w$ has exactly two neighbors $a,b$. Blowing up $w$ by a clique introduces many additional observed covariance entries, but this clique interacts with the rest of the graph only through $a,b$. Assume $ab$ is not an edge. Then the only changes in the equations of our chosen chordal completion of the blown up graph is possible at the entry $ab$ because all the new observations introduced in the blow-up are not fill variables. Recalling the Schur formula of the entry at a special edge from \Cref{sec:specialization-one-fill}, the only effective generic data are given by the ``$c$'' and the denominator polynomial ``$D(x)$'' from there. It thus reduces to the fact that the corresponding likelihood with generic data $c,D(c)$ dictate the ML degree of each blowup because each blowup is simply an extension of scalars. We need the following form of this latter fact, which will be used to prove the main result using the strategy above.

\begin{lemma}\label{lem:extension-of-scalars}
Let $\iota:K\to L$ be an embedding of fields. Suppose $h\in R\sett K[(z_{i})_{i\in[n]}]\smallsetminus \set{0}$ and $\mathcal P = \set{P_{1},\cdots,P_{m}} \subseteq R[h^{-1}]$. Define $\mathcal A \sett R[h^{-1}]/(\mathcal P), \mathcal A_{L} \sett L[(z_{i})_{i\in[n]}, \iota(h)^{-1}]/(\iota(\mathcal P))$. Then there is a natural $L-$algebra isomorphism $\mathcal A_{L} \cong L\otimes_{K,\iota}\mathcal A$. If either algebra is finite dimensional over their respective fields, we have $\dim_{K}\mathcal A = \dim_{L}\mathcal A_{L}$.
\end{lemma}
\begin{proof} Extension of scalars commutes with forming polynomial rings and with localization. Consequently, $L\otimes_{K,\iota}R[h^{-1}] \cong L[(z_{i})_{i\in[n]},\iota(h)^{-1}]$. $\iota(h)\ne 0$ because $\iota$ is a monomorphism of fields. Passing to the quotient by the images of $P_1,\cdots,P_m$ gives
\[
L\otimes_{K,\iota}\mathcal A \cong \frac{L[(z_{i})_{i\in[n]},\iota(h)^{-1}]}{(\iota(P_{j}):j\in[m])} = \mathcal A_{L}.\]
Let $\mathcal E$ be a $K-$basis of $\mathcal A$. Note, we don't yet assume $\mathcal E$ is finite. As $L-$vector spaces, we have
\[
L\otimes_{K,\iota}\mathcal A \cong L\otimes_{K,\iota} \left(\bigoplus_{e\in\mathcal E}Ke\right) \cong \bigoplus_{e\in\mathcal E}L.
\]
Thus $\mathcal A_{L}$ is finite-dimensional precisely when $\mathcal E$ is finite, and its dimension is then $\abs {\mathcal E}=\dim_{K}\mathcal A$.
\end{proof}

\begin{theorem}
\label{thm:degree-two-blow-up}
If $w\in V(G)$ has $\deg w = 2$, then $\MLdeg(\bl_w^q(G))=\MLdeg(G)\fa q\geq1$.
\end{theorem}

\begin{proof}
Let $G,w$ be as given with $S \sett \cN_{w,G}=\set{a,b}$. If $ab\in E(G)$ then $w$ is simplicial, and the result follows from \Cref{prop:simplicial-blow-up}. Henceforth assume $ab\notin E(G)$. 

Take $V_{0} = V(G) \smallsetminus \set w, G_{0} = G[V_{0}] = G\smallsetminus w$ and $G_{q}\sett \bl_{w}^{q}(G)$. Let $Q_{q}=K_{q}$ be the part that replaces $w$ in $G$. If $V_{0}=S$ then $G$ is a chordal graph whence $G_{q} = K_{Q_{q}\cup \set{a}} \cup_{Q_{q}} K_{Q_{q}\cup\set b}$ whence $\MLdeg(G_{q}) = 1 =\MLdeg(G)$. Therefore assume $\abs{V_{0}}\ge 3$. We choose the chordal completion $H_{q} = K_{V_{0}} \cup_{S} K_{S\cup Q_{q}}$ of $G_{q}$. Then $F = {{V_{0}}\choose 2}\smallsetminus E(G_{0})$ determine the fill edges, that is, $H_{q}=G_{q}+F$. Clearly, $G_{0}+F = K_{V_{0}} \eqqcolon H_{0}$. In particular, $ab\in F$. This means that we have the same set of fill edges $F$ for every $q$. Clearly $\cC(H_{q}) = \set{V_{0}, S\cup Q_{q}}$ and thus the only clique tree $T_{q}$ of $H_{q}$ has $\cS(H_{q},T_{q}) = \set{\set{a,b}} = \set S$. By the chordal formula (and \Cref{lem:clique-gluing-identity}), \begin{equation}\label{eq:full-chordal}K^{H_{q}} = \lb K^{H_{0}}\rb^{0} + \lb \Sigma_{S\cup Q_{q}}^{-1}\rb^{0} - \lb\Sigma_{S}^{-1}\rb^{0}\end{equation} where the matrices have been embedded to the size of $V(G_{q})\times V(G_{q})$ padded with zeroes. We thus want to localize (for later) at $\Delta_{H_{q}} = \det (K^{H_{0}})^{-1}\det \Sigma_{S\cup Q_{q}} \det \Sigma_{S}$.

Introduce a variable $x_{f}$ for every $f\in F$ and write $x=x_{ab}$ for the special $ab\in F$. Let $\eta = (s_{ij})_{ij\in \cO_{G_{0}}}$ and $k\sett \C(\eta)$. $\eta$ comprises the observed entries in $G_{0}$. Let $\Sigma(x_{F})$ be the symmetric covariance matrix on $G_{0}$. Then \[\Sigma_{S}(x) = \begin{bmatrix}s_{aa} & x\\x&s_{bb}\end{bmatrix},\qquad \Sigma_{S\cup Q_{q}}(x) = \begin{bmatrix}
s_{aa} & x & r^{\top}\\
x & s_{bb} & u^{\top}\\
r & u & T
\end{bmatrix}.
\]
The entries of $r,u,T$ are precisely the additional observed entries of $G_{q}$. The data field of $G_{q}$ is $L_{q}\sett K_{G_{q}} = k\lp\set{r_{i}}_{i\in [q]} \cup \set{u_{i}}_{i\in [q]} \cup \set{t_{ij}}_{1\le i\le j \le q}\rp$. Consider the quantities \[\tilde c \sett r^{\top}T^{-1}u \in L_{q}, \qquad \tilde D(x)\sett (s_{aa}-r^{\top}T^{-1}r)(s_{bb}-u^{\top}T^{-1}u) - (x-\tilde c)^{2}\in L_{q}[x].\] By the Schur complement determinant formula discussed earlier, $\det \Sigma_{S\cup Q_{q}} = \det T\cdot \tilde D(x)$. The Schur complement inverse formula gives $(\Sigma_{S\cup Q_{q}}^{-1})_{ab} = -(x-\tilde c)/\tilde D(x)$. The quantities $\tilde c$ and $\tilde D(\tilde c)$ are algebraically independent over $k$. Indeed, specializing $T=I_q$, $r=(r_1,0,\ldots,0)^\top$, and $u=(u_1,0,\ldots,0)^\top$ gives $\tilde c = r_1u_1, \tilde D(\tilde c) = (s_{aa}-r_1^2)(s_{bb}-u_1^2)$.
Their Jacobian determinant with respect to $r_1,u_1$ is
$2(s_{bb}r_1^2-s_{aa}u_1^2)$, which is not identically zero.
Thus this specialization defines a dominant map to the two effective
parameters, proving the claimed algebraic independence for every
$q\ge1$.  Consequently there is a $k-$embedding $\iota_{q}: k(c,D) \hookrightarrow L_{q}$ given by $c\mapsto \tilde c, D\mapsto \tilde D(\tilde c)$.

The likelihood ideal for any $G_{q}$ is cut out by the following rational functions in $L_{q}(x_{F})$ for $f\in F$:
\[p_{f}(x_{F}) = \begin{cases}
(\Sigma_{H_{0}}(x_{F})^{-1})_{f} & \text{ if } f\ne ab\\
(\Sigma_{H_{0}}(x_{F})^{-1})_{ab} - (x-\tilde c)/\tilde D(x) + x/(s_{aa}s_{bb}-x^{2}) & \text{ if } f = ab
\end{cases}.
\] 
This directly comes from \eqref{eq:full-chordal} and the fact that $F\cap {{S\cup Q_{q}}\choose 2} = \set{ab}$. Therefore take $h = \det \Sigma_{H_{0}}(x_{F}) \cdot (D-(x-c)^{2})\cdot (s_{aa}s_{bb}-x^{2}) \in k(c,D)[x_{F}]$ and define \[
P_{f}(x_{F}) \sett \begin{cases}
(\Sigma_{H_{0}}(x_{F})^{-1})_{f} & \text{ if } f\ne ab\\
(\Sigma_{H_{0}}(x_{F})^{-1})_{ab} - (x- c)/(D-(x-c)^{2}) + x/(s_{aa}s_{bb}-x^{2}) & \text{ if } f = ab
\end{cases} 
\]
in $k(c,D)[x_{F},h^{-1}]$. Take $\mathcal A = k(c,D)[x_{F},h^{-1}]/(P_{f}: f\in F)$. Note that $\Sigma_{H_{0}}(x_{F})$ has nonzero determinant in $k[x_{F}]$. Also $D-(x-c)^{2}$ and $s_{aa}s_{bb}-x^{2}$ are nonzero in $k(c,D)[x_{F}]$. Therefore $h\ne 0$. Therefore the localization, and hence the algebra $\mathcal A$, is well defined. Note that $p_{f} = \iota_{q}(P_{f})\fa f\in F$ and $\iota_{q}(h) = \det \Sigma_{H_{0}}(x_{F}) \cdot \tilde D(x)\cdot (s_{aa}s_{bb}-x^{2})$ which is just $\Delta_{H_{q}}$ up to units in $L_{q}$. The generic likelihood algebra of $G_{q}$ is $\mathcal A_{q} = L_{q}[x_{F},\Delta_{H_{q}}^{-1}]/(p_{f}:f\in F)$. By \Cref{lem:extension-of-scalars}, \[
\dim_{k(c,D)}\mathcal A = \dim_{L_{q}}\mathcal A_{q} = \MLdeg(G_{q}).
\]
Since $\MLdeg(G_{q}) = \dim_{k(c,D)}\mathcal A$ and $k(c,D),\mathcal A$ are independent of $q$ it must happen that $\MLdeg(G_{q}) = \MLdeg(G_{1}) = \MLdeg(G)$.
\end{proof}

\begin{corollary}\label{cor:2-regular}
Let $v_1,\ldots,v_m$ be pairwise nonadjacent degree-two vertices of a graph $G$, and let $q_1,\ldots,q_m\geq1$. Then $\MLdeg\left( \bl_{v_m}^{q_m}\circ\cdots\circ\bl_{v_1}^{q_1}(G) \right)=\MLdeg(G)$. In particular, if $G=C_n$ with $n\geq3$, then $\MLdeg \lp \bl_{v_m}^{q_m}\circ\cdots\circ\bl_{v_1}^{q_1}(C_n) \rp = (n-3)2^{n-2}+1.$
\end{corollary}

\begin{proof}
Apply \Cref{thm:degree-two-blow-up} successively. If $i\neq j$, then $v_i$ is not adjacent to $v_j$. Replacing $v_j$ by a clique therefore adds no new neighbor to $v_i$ and removes no neighbor of $v_i$, hence $v_i$ remains of degree two after any earlier blow-ups. This proves the first identity by induction. The second follows from the \ref{eq:dvdss} which says  $\MLdeg(C_n)=(n-3)2^{n-2}+1$ for $n\geq4$; for $n=3$, the graph is complete and both sides equal $1$.
\end{proof}

\subsection{Conclusion and further questions}\label{sec:questions}

This work shows that chordal completion can retain substantial information about the likelihood geometry of a nonchordal Gaussian graphical model. When a single added edge makes the graph chordal, the likelihood equations reduce to one rational equation, and its degree is determined entirely by the maximal cliques and distinct clique-tree separators containing that edge. For several fill edges, the interaction graph identifies the parts of the reduced likelihood system that genuinely depend on one another: its connected components factor the likelihood algebra, and isolated fill edges yield a product of the corresponding one-edge degrees. The clique-blow-up results further show that maximum likelihood degree is insensitive to several natural local enlargements of a graph, although the six-vertex counterexample demonstrates that such invariance is not universal. Together, these results suggest that the algebraic complexity of likelihood estimation is governed less by the total size of a chordal completion than by the local pattern in which its fill edges occur among cliques and separators. A natural next step is therefore to understand the likelihood algebras associated with connected components of the interaction graph, particularly for two or three interacting fill edges, and to determine whether these local systems can support practical exact or numerical methods for models admitting sparse chordal completions.

Here are some open questions.

\begin{question}[Interacting components]\label{qs:interacting-edge}
Let $H$ be chordal, let $T$ be a clique tree, and let $F\subseteq E(H)$ have connected interaction graph $\Gamma_T(F)$. Can the length of
\[
 K^H_f(x_F)=0,\qquad f\in F,
\]
away from the clique and separator determinants, be computed from the bags containing the edges of $F$?
\end{question}

For one fill edge, \Cref{thm:one-fill} gives the answer. The graph $K_4\smallsetminus\{12,34\}\cong C_4$ shows that two interacting edges can already contribute degree $5$.

\begin{question}[Two interacting fill edges]
Find ML degree formulae when exactly two fill edges interact in a chordal completion.
\end{question}

\begin{question}[Choice of chordal completion]
If $G$ has several chordal completions with few fill edges, how do the resulting interaction graphs and local systems compare? Is there a completion that minimizes the size of the largest interacting component, or otherwise minimizes the algebraic complexity of the component systems in \Cref{lem:component-factorization}?
\end{question}

\begin{question}[$K_{3,m}$ problem {\cite{Uhler_mle_ggm}}]\label{qs:k3m}
Find a closed formula or recurrence for $\MLdeg(K_{3,m})$.
\end{question}

\section{Acknowledgement}
ChatGPT 5.6 Plus was used for organizing the paper, language corrections and searching for relevant MSC codes. I also thank James Saunderson, Yong Sheng Soh and Oscar Leong for helpful comments.

\newpage
\nocite{M2}
\bibliography{references}
\bibliographystyle{siamplain}

\newpage
\appendix

\section{Supplementary lemmas}

\begin{lemma}[Dominance and generic finiteness of the likelihood projection]
\label{lem:likelihood-projection-generically-finite}
The coordinate projection $\pi_G:\cM_G^\circ\to\C^{\cO_G}$ is dominant and generically finite.
\end{lemma}

\begin{proof}
In concentration coordinates, $\pi_G$ is represented by
\[L_G^\times\to\C^{\cO_G}, \qquad K\longmapsto\pi_G(K^{-1}).\]
Its differential at $K=I$ is $L\mapsto-\pi_G(L).$ Under the natural coordinate identification $L_G\cong\C^{\cO_G}$, this is the negative identity map. Therefore it has rank $\abs{\cO_G}$, proving that $\pi_G$ is dominant. Since its source and target are irreducible varieties of the same dimension $\abs{\cO_G}$, the fiber-dimension theorem implies that $\pi_G$ is generically finite.
\end{proof}

\begin{lemma}[Gluing identity]\label{lem:clique-gluing-identity}
Let $A,S,B$ be pairwise disjoint index sets. Let
\[\Sigma_1=\begin{bmatrix}P&U\\U^\top&R\end{bmatrix}_{A,S}, \qquad \Sigma_2=\begin{bmatrix}R&V\\V^\top&Q\end{bmatrix}_{S,B}\]
be invertible symmetric matrices with $R$ invertible. Then
\begin{equation}\label{eq:glued-covariance}
\hat\Sigma=
\begin{bmatrix}
P&U&UR^{-1}V\\
U^\top&R&V\\
V^\top R^{-1}U^\top&V^\top&Q
\end{bmatrix}_{A,S,B}
\end{equation}
is invertible and satisfies
\begin{align}
\hat\Sigma^{-1} &=\lb\Sigma_1^{-1}\rb^{0}+\lb \Sigma_2^{-1}\rb^{0}-\lb R^{-1}\rb^{0}, \label{eq:glued-inverse}\\
\det\hat\Sigma &=\frac{\det\Sigma_1\det\Sigma_2}{\det R}. \label{eq:glued-determinant}
\end{align}
Moreover, \eqref{eq:glued-covariance} is the unique nonsingular symmetric matrix whose restrictions to $A\cup S$ and $S\cup B$ are $\Sigma_1$ and $\Sigma_2$ respectively, with zero entries in the $A\times B$ block of the inverse.
\end{lemma}

\begin{proof}
We denote the Schur complement of a block $X$ in matrix $M$ as $S_{X}(M)$. First note that $\Sigma_{1},\Sigma_{2}$ are invertible and the blocks $R$ within them are invertible. So  $S_{R}(\Sigma_{1}) = P-UR^{-1}U^{\top}, S_{R}(\Sigma_{2}) = Q-V^{\top}R^{-1}V$ are invertible. Now reorder the matrix $\hat\Sigma$ as $A,B,S$ and take the Schur complement of $R$ which is precisely
\begin{equation}\label{eq:diag-block}
S_{R}(\hat\Sigma) = \begin{bmatrix}
P-UR^{-1}U^\top&0\\
0&Q-V^\top R^{-1}V
\end{bmatrix} = 
\begin{bmatrix}
S_{R}(\Sigma_{1})\\&S_{R}(\Sigma_{2})
\end{bmatrix}.
\end{equation}
The determinant formula for a Schur complement says \[\det\hat\Sigma = \det R \cdot \det S_{R}(\hat\Sigma) = \det R\cdot \det S_{R}(\Sigma_{1}) \det S_{R}(\Sigma_{2}) = \frac{\det \Sigma_{1}\det\Sigma_{2}}{\det R}.\]

Now we use the Schur complement inverse formula. Order the entries of  $\lb\Sigma_{1}\rb^{0}, \lb\Sigma_{2}\rb^{0}, \hat\Sigma$ as $A,B,S$. Then the Schur complement inverses, with respect to $R$, for $[\Sigma_{1}]^{0}, [\Sigma_{2}]^{0}, \hat\Sigma^{-1}$ are
\begin{align*}
(\Sigma_{1}^{-1})^{0} &= \begin{bmatrix}
S_{R}(\Sigma_{1})^{-1}&0 & -S_{R}(\Sigma_{1})^{-1}UR^{-1}\\
0&0&0\\
-R^{-1}U^{\top}S_{R}(\Sigma_{1})^{-1}&0& R^{-1}+ R^{-1}U^{\top} S_{R}(\Sigma_{1})^{-1}UR^{-1}
\end{bmatrix}_{A,B,S}\\
(\Sigma_{2}^{-1})^{0} &= \begin{bmatrix}
0 & 0 & 0\\
0& S_{R}(\Sigma_{2})^{-1} & -S_{R}(\Sigma_{2})^{-1}V^{\top}R^{-1}\\
0& -R^{-1}VS_{R}(\Sigma_{2})^{-1} & R^{-1}+ R^{-1}V S_{R}(\Sigma_{2})^{-1}V^{\top}R^{-1}
\end{bmatrix}_{A,B,S}\\
\hat\Sigma^{-1} &= \begin{bmatrix}
S_{R}(\hat\Sigma)^{-1} & -S_{R}(\hat\Sigma)^{-1} \begin{bmatrix}U\\V^{\top}\end{bmatrix}R^{-1}\\
-R^{-1}\begin{bmatrix}U^{\top}&V\end{bmatrix}S_{R}(\hat\Sigma)^{-1} & R^{-1} + R^{-1}\begin{bmatrix}U^{\top}&V\end{bmatrix}S_{R}(\hat\Sigma)^{-1}\begin{bmatrix}U\\V^{\top}\end{bmatrix}R^{-1}
\end{bmatrix}_{A,B,S}.
\end{align*}
Recalling the special diagonal-block structure in \eqref{eq:diag-block} gives \eqref{eq:glued-inverse}.

For uniqueness, let $X$ be any common extension of $\Sigma_1$ and $\Sigma_2$, and write its unknown $A\times B$ block as $Z$. The Schur complement of $R$ in $X$, again ordered as $(A,B,S)$, has off-diagonal block $Z-UR^{-1}V$. The $A\times B$ block of $X^{-1}$ vanishes if and only if this Schur complement, and hence its inverse, is block diagonal. Thus it vanishes if and only if $Z=UR^{-1}V$.
\end{proof}

\begin{proof}[Proof of \Cref{prop:chordal-formula}]
We induct on $\abs{\cC(H)}$. The claim is immediate when $H$ has one maximal clique. Otherwise choose a leaf clique $C$ of $T$, let $C'$ be its neighbor, put $S=C\cap C'$, $A=C\smallsetminus S$, and $H'=H[V(H)\smallsetminus A]$. The running-intersection property implies that every vertex of $A$ occurs only in $C$. Hence there is no edge of $H$ between $A$ and $V(H')\smallsetminus S$, and deleting $C$ from $T$ gives a clique tree of $H'$.

By induction, the partial matrix on $H'$ has a unique completion $\hat\Sigma^{H'}$, whose inverse is the clique-separator sum for $T\smallsetminus C$. Glue $y_C$ and $\hat\Sigma^{H'}$ along their common nonsingular submatrix $y_S$. By \Cref{lem:clique-gluing-identity}, the resulting matrix is nonsingular, matches every specified entry of $y$, and its inverse is
\[\lp y_C^{-1}\rp^{0}+\lp\lp\hat\Sigma^{H'}\rp^{-1}\rp^{0} - \lp y_S^{-1}\rp^{0},\]
which is exactly \eqref{eq:chordal-formula}. It has zeros at all nonedges: the induction hypothesis handles $H'$, and the gluing identity gives a zero block between $A$ and $V(H')\smallsetminus S$.

Conversely, let $\tilde\Sigma\in\mathcal M_H^\circ$ have the prescribed entries, and put $B=V(H')\smallsetminus S$. Take the Schur complement of $\tilde\Sigma_S=y_S$ in $\tilde\Sigma$, ordered as $(A,B,S)$. Its inverse is the $(A\cup B)$-principal block of $\tilde\Sigma^{-1}$, whose $A\times B$ block is zero because there are no edges between $A$ and $B$. Hence that inverse, and therefore the Schur complement itself, is block diagonal. This forces
\[\tilde\Sigma_{A,B}=y_{A,S}y_S^{-1}\tilde\Sigma_{S,B}\]
and also shows that $\tilde\Sigma_{H'}$ is nonsingular.  Now, \Cref{lem:clique-gluing-identity} gives
\[\tilde\Sigma^{-1} = \lp y_C^{-1}\rp^{0} + \lp\tilde\Sigma_{H'}^{-1}\rp ^{0}-\lp y_S^{-1}\rp^{0}.\]
Consequently, $\tilde\Sigma_{H'}^{-1}$ vanishes at every nonedge of $H'$. Indeed, every such nonedge has at least one endpoint outside the clique $S$, where the other two displayed summands vanish. The induction hypothesis therefore forces $\tilde\Sigma_{H'}=\hat\Sigma^{H'}$, and then the cross-block identity forces $\tilde\Sigma = \hat\Sigma^H$. This proves uniqueness. Finally, iterating \eqref{eq:glued-determinant} and taking reciprocals gives \eqref{eq:chordal-det-factorization}
\end{proof}

\begin{lemma}[Saturation of the determinant critical ideal]
\label{lem:determinant-saturation-localization}
Let $\mathcal E$ be an index set, $K$ be a field of characteristic zero, $R=K[x_{ij}: ij\in \mathcal E],$ and $A$ be a symmetric matrix such that $A_{ij}=x_{ij}\fa ij\in \mathcal E$ and $A_{ij}\in K$ otherwise. Write
\[p=\det(A),\qquad J=(\cof_{ij}(A):(i,j)\in \mathcal E), \qquad I=J:p^\infty.\]
Then the natural $K$-algebra homomorphism $R/I\longrightarrow R[p^{-1}]/JR[p^{-1}]$ is an isomorphism.
\end{lemma}

\begin{proof}
Write $C_{ij} = \cof_{ij}(A)$. For a diagonal variable $x_{ii}$, Jacobi's cofactor formula gives $\partial p/\partial x_{ii}=C_{ii}$. For $i<j$, the variable $x_{ij}$ occurs in both the $(i,j)$ and $(j,i)$-entries of $A$. Since $A$ and its cofactor matrix are
symmetric,
\[\frac{\partial p}{\partial x_{ij}} = C_{ij}+C_{ji}=2C_{ij}.\]
Because $\operatorname{char}K=0$, it follows that $J= \lp \partial_{ij} p:ij\in \mathcal E\rp$. Thus $J$ is the Jacobian ideal of $p$ with respect to all the variables of $R$. By the  saturation-localization correspondence,
\[J:p^\infty=JR[p^{-1}]\cap R, \qquad JR[p^{-1}]=IR[p^{-1}]. \]
Consequently,
\begin{equation}\label{eq:isomlocal}
R[p^{-1}]/JR[p^{-1}] \cong (R/I)[p^{-1}].
\end{equation}

We will be done if we show that the image of $p$ is a unit in $R/I$. By the definition of saturation, $I:p=I,$ so $p$ is a nonzerodivisor on $R/I$. Since $R/I$ is Noetherian, every minimal prime of $R/I$ is associated. Hence $p$ belongs to no minimal prime over $I$.

Let $\mathfrak q$ be a minimal prime over $I$ and put $F=\operatorname{Frac}(R/\mathfrak q)$. Since $J\subseteq\mathfrak q$,  
\[d\overline p = \sum_{ij\in \mathcal E}\frac{\partial\overline p}{\partial x_{ij}}d x_{ij} = 0\] 
in $\Omega_{F/K}$, the vector space of K\"ahler differentials. By \cite[Theorem~$87(i)$, p.~$267$]{MatsumuraCA}, for an extension $F/K$ of fields of characteristic zero, if $u\in F$ then one has $du=0$ in $\Omega_{F/K}$ iff $u$ is algebraic over $K$. Thus $\overline p$ is algebraic over $K$. Since $p\notin\mathfrak q$, it is nonzero in $F$, and hence it is a unit already in $R/\mathfrak q$. In fact, this inverse is a polynomial in $\overline p$ with coefficients in $K$ (constructed using the minimal polynomial).

Therefore $p$ is a unit modulo every minimal prime over $I$. Every maximal ideal containing $I$ contains a minimal prime over $I$, so no such maximal ideal can contain $p$. Hence $I+(p)=R$, meaning that $p$ is a unit in $R/I$. It follows that
\[(R/I)[p^{-1}]=R/I.\]
Combining this with \eqref{eq:isomlocal} proves the result.
\end{proof}

\begin{lemma}[Clique--separator determinants are units on the generic fiber]
\label{lem:Delta-unit}
Let $H$ be chordal, let $F\subseteq E(H)$, and put $G=H\setminus F$. Let $\mathcal A_{G}\sett R_{G}/I_{G}$ be the generic likelihood algebra defined in \eqref{eq:likelihoodequations}. Identify the nonedge variables corresponding to $F$ with the variables $x_{f}$, and let
\[\Delta_{H}(x_{F})= \prod_{C\in\mathcal C(H)}\det y_{C}(x_{F}) \prod_{e\in E(T)}\det y_{B_{e}}(x_{F}).\]
Then the image of every factor $\det y_{C}(x_{F})$ and $\det y_{B_{e}}(x_{F})$ in $\mathcal A_G$ is a unit. In particular, the image of $\Delta_{H}(x_{F})$ is a unit in $\mathcal A_{G}$.
\end{lemma}

\begin{proof}
Let
\[X=\mathcal M_{G}^\circ,\qquad Y=\mathbb A^{\cO_{G}}, \]
and let $\pi_{G}:X\to Y$ be the coordinate projection. By \Cref{lem:likelihood-projection-generically-finite}, $\pi_{G}$ is dominant and generically finite. Matrix inversion identifies $X$ with a principal open subset of the affine space $L_G$. Consequently, $X$ is integral. Let $A_X=\mathbb C[X] \cong \C[k_{ij} : ij\in \mathcal \cO_{G}][(\det K)^{-1}]$ and $A_Y=\C[Y] \cong \C[s_{ij}: ij\in \mathcal \cO_{G}]$. Dominance of $\pi_G$ implies that the induced pullback homomorphism
\[\pi_{G}^{*}:A_{Y}\longrightarrow A_{X}, \qquad s_{ij}\mapsto (K^{-1})_{ij}\]
is injective. Recall that $\K_{G}=\operatorname{Frac}(A_{Y})$, so the coordinate algebra of the generic fiber is
\[A_{X}\otimes_{A_{Y}}\K_{G} = (A_{Y}\setminus\set{0})^{-1}A_{X}.\]
Since $A_{X}$ is a domain, this localization is also a domain. But $\pi_{G}$ is generically finite, whence $\operatorname{Frac}(A_{X})/\K_{G}$ is a finite extension. Consequently $(A_{Y}\setminus\set{0})^{-1}A_{X}$ is finite-dimensional over $\K_{G}$ (contained in $\operatorname{Frac}(A_{X})$). A finite-dimensional algebra over a field that is also a domain is a field. By \eqref{eq:likelihoodequations}, this generic-fiber algebra is naturally isomorphic to $\mathcal A_{G}$.

For each clique $C\in\cC(H)$, consider the regular function $\det_C:X\to \C, \det_{C}(\Sigma) = \det(\Sigma_{C})$. Since $\det_{C}(I)=1$, the function $\det_{C}$ is a nonzero element in the domain $A_{X}$. Since localization of a domain is injective, the image of $\det_{C}$ in $(A_{Y}\smallsetminus\set0)^{-1}A_{X} \cong \mathcal A_{G}$ is also nonzero and is therefore a unit. This image is precisely $\det y_{C}(x_{F})$.

The same argument applies to every separator $B_{e}$, because $\det(I_{B_{e}})=1$. Thus every clique and separator determinant is a unit in $\mathcal A_{G}$, and hence so is their product $\Delta_{H}(x_{F})$.
\end{proof}

\section{Proof of \Cref{prop:double-blow-up-degree}}\label{sec:code}

Let $D$ be the graph in \Cref{fig:counter-example-blow-up}. We first prove that the degree is generically $9$. The exact Macaulay2 calculation below is an independent check at one rational specialization.

Its covariance matrix looks like \[
\Sigma(x,y,\alpha,\beta) = \Sigma^{D} = \begin{bmatrix}
*&*&*&*&x&*\\
*&*&*&*&y&*\\
*&*&*&*&*&\alpha\\
*&*&*&*&*&\beta\\
x&y&*&*&*&*\\
*&*&\alpha&\beta&*&*
\end{bmatrix}
\]
where the $*$'s denote independent generic observed data, modulo the constraint that this matrix is symmetric. Recall Jacobi's formula for a symmetric matrix $A$ and $i\ne j$: \[\partial \det A/\partial A_{ij} = 2\det A\cdot (A^{-1})_{ij} = 2\cof_{ij}(A).\]
The likelihood equations are \[\cof_{15}\Sigma = \cof_{25}\Sigma = \cof_{36}\Sigma = \cof_{46}\Sigma = 0.\]
Further, if we treat $x,y$ as fixed numbers, that matrix corresponds to the graph $H\sett D\cup\set{15,25}$ which is chordal. Therefore, fixing $x,y$ (generically) determines $\alpha,\beta$. This is clear because the chordal formula gives $K^{H} = [\Sigma_{C_{1}}^{-1}]^{0} + [\Sigma_{C_{2}}^{-1}]^{0} - [\Sigma_{S}^{-1}]^{0}$ and none of the $C_{1}, C_{2}, S$ contain the edges $36, 46$ whence these entries in $K^{H}$ are automatically zero. Moreover, $\cC(H) = \set{C_{1},C_{2}}$ where $C_{1}=\set{1,2,3,4,5}, C_{2}=\set{1,2,5,6}$. The only possible tree structure is to join $C_{1},C_{2}$ with separator $S=C_{1}\cap C_{2} = \set{1,2,5}$. The chordal formula gives $\det \Sigma^{H} = \det \Sigma_{C_{1}}\det \Sigma_{C_{2}}/\det\Sigma_{S}$. Take $q_{1}(x,y)\sett \det \Sigma_{C_{1}}, q_{2}(x,y)\sett \det \Sigma_{C_{2}}, q_{0}(x,y) \sett \det \Sigma_{S}$. Observe that these are quadratic in $x,y$ and are independent of $\alpha,\beta$.  This reduces the equations to $\cof_{15}\Sigma = \cof_{25}\Sigma=0$. Chordal equation gives $K^{H}_{f} = (\Sigma_{C_{1}}^{-1})_{f}+(\Sigma_{C_{2}}^{-1})_{f}-(\Sigma_{S}^{-1})_{f} = 0$ for $f\in\set{15,25}$. Jacobi's formula gives \[\partial_{x}\log(q_{1}q_{2}/q_{0}) = \partial_{y}\log(q_{1}q_{2}/q_{0}) = 0.\]
This is because, for example, $2(\Sigma_{C_{1}}^{-1})_{15} = \frac1{q_{1}}\partial q_{1}/\partial (\Sigma_{C_{1}})_{15} = \partial_{x} q_{1}/q_{1} = \partial_{x}\log q_{1}$. This gives the following polynomial one-form:
\[
\Omega \sett F_{x}\d x+ F_{y} \d y = q_{0}q_{2}\d q_{1} + q_{0}q_{1}\d q_{2} - q_{1}q_{2}\d q_{0} = q_{0}q_{1}q_{2}~\d \log(q_{1}q_{2}/q_{0}).
\] 
We want to count the roots of $F_{x} = F_{y} = 0$ on the open set $\set{q_{0}q_{1}q_{2}\ne 0}$. First we claim that the mutual intersections of $q_{0},q_{1},q_{2}$ are the following:
\begin{itemize}
\item $q_{0}=q_{1}=0$ meet transversally at four distinct points with multiplicity one each;
\item $q_{1}=q_{2}=0$ meet transversally at four distinct points with multiplicity one each;
\item $q_{0}=q_{2}=0$ meet at two distinct points with multiplicity two each;
\item $q_{0}=q_{1}=q_{2}=0$ has no solution.
\end{itemize}

\begin{claim}
$q_{0}=q_{1}=0$ meet at four distinct points with multiplicity one each. Same is true for $q_{1}=q_{2}=0$.
\end{claim}
\begin{proof}
We show this only for $q_{0}=q_{1}=0$ and the other case follows similarly.
Recall that $q_{0} = \det \Sigma_{S} = \det \Sigma_{\set{1,2,5}}, q_{1} = \det \Sigma_{C_{1}} = \det \Sigma_{\set{1,2,3,4,5}}$. Note that \[
\Sigma_{C_{1}} = \begin{bmatrix}
 &&&s_{13} & s_{14}\\
&\Sigma_{S}&& s_{23} & s_{24}\\
&&& s_{35} & s_{45}\\
s_{13} & s_{23} & s_{35} & s_{33} & s_{34}\\
s_{14} & s_{24} & s_{45} & s_{34} & s_{44}
\end{bmatrix}_{(1,2,5,3,4)},\qquad  \Sigma_{S} = \begin{bmatrix}
s_{11} & s_{12} & x\\
s_{12} & s_{22} & y\\
x & y & s_{55}
\end{bmatrix}_{(1,2,5)}.
\] up to reordering of indices. By Schur complement determinant formula, $q_{1} = \det \Sigma_{\set{3,4}} \cdot \det(\Sigma_{S} - W\Sigma_{\set{3,4}}^{-1}W^{\top})$ where $W$ is the $3\times 2$ submatrix on the top right of $\Sigma_{C_{1}}$ indexed by columns $3,4$ and rows $1,2,5$.

We will show that the system $q_{0}=q_{1}=0$ has four distinct intersection points for generic data. This condition can be expressed as the Zariski open condition \[c_{y^{2}}(q_{0})c_{y^{2}}(q_{1})c_{x^{4}}(\text{Res}_{y}(q_{0},q_{1}))\text{Disc}_{x}(\text{Res}_{y}(q_{0},q_{1}))\ne 0\]
where $c_{z^{i}}(p)$ extracts the coefficient of the monomial $z^{i}$ from the polynomial $p$. Indeed, $c_{x^{4}}(\text{Res}_{y}(q_{0},q_{1})) \ne 0$ makes $\text{Res}_{y}(q_{0},q_{1})$ a quartic polynomial, $\text{Disc}_{x}(\text{Res}_{y}(q_{0},q_{1})) \ne 0$ makes its four roots distinct, and finally $c_{y^{2}}(q_{0})c_{y^{2}}(q_{1}) \ne 0$ ensures that each root of $\text{Res}_{y}(q_{0},q_{1})$ gives a common $y$-root for each of the four roots of $\text{Res}_{y}(q_{0},q_{1})$. By Bezout's theorem, there can be at most four roots in $(x,y)$, and the above open condition ensures that the four roots are distinct, thus have multiplicity one each. Since this is an open condition, it is enough to specialize to one data specialization which satisfies the above open condition. Here is the verification in \texttt{M2}.

\begin{lstlisting}[language=Macaulay2]
kk = QQ;
R = kk[x,y];
Sigma0 = matrix {
   {-1, 2, x},
   {2, -2, y},
   {x, y, -5}
};
Sigma1 = matrix{
   {-1, 2, 3, 4, x},
   {2, -2, 6, 7, y},
   {3, 6, -3, 9, 10},
   {4, 7, 9, -4, 11},
   {x, y, 10, 11, -5}
};
q0 = det Sigma0; 
q1 = det Sigma1;
cy20 = coefficient(y^2, q0);
cy21 = coefficient(y^2, q1);
Res = resultant(q0,q1,y);
D = discriminant(Res,x);
(cy20 != 0, cy21 != 0, degree Res, D != 0)
\end{lstlisting}

The outputs are \begin{align*}
\texttt{cy20 != 0} &: \qquad\texttt{true}, \\
\texttt{cy21 != 0} &: \qquad\texttt{true}, \\
\texttt{degree Res} &: \qquad\texttt{\set{4}},\\
\texttt{D != 0} &: \qquad \texttt{true}.
\end{align*}
\end{proof}

\begin{claim}\label{cl:tangent}
$q_{0}=q_{2}=0$ meet at two distinct points with multiplicity two each. 
\end{claim}
\begin{proof}
Recall $q_{0} = \det \Sigma_{S} = \det \Sigma_{\set{1,2,5}}, q_{2} = \det \Sigma_{C_{2}} = \det \Sigma_{\set{1,2,5,6}}$. Note that \[
\Sigma_{C_{2}} = \begin{bmatrix}
A & w & b\\
w^{\top} & d & e \\
b^{\top} & e & t
\end{bmatrix}_{(1,2,5,6)}, \qquad  
\Sigma_{S} = \begin{bmatrix}
A & w\\
w^{\top} & d
\end{bmatrix}_{(1,2,5)}
\]
where $A$ is a generic symmetric $2\times 2$ matrix, $b$ is a generic $2\times 1$ vector, $d,e,t$ are generic scalars and $w$ is the $2\times 1$ vector with entries $x,y$. Let \[
h=d-w^{\top}A^{-1}w, \qquad \ell = e-b^{\top}A^{-1}w, \qquad \kappa = t-b^{\top} A^{-1}b, \qquad \Delta = \det A.
\] 
First restrict to the parameter space open set $\Delta\kappa\ne 0$ (which is true generically). Then by Schur determinant formula (with respect to $A$ in $\Sigma_{C_{2}}$), $q_{2} = \det \Sigma_{C_{2}} = \kappa q_{0} - \Delta \ell^{2}$. Therefore $(q_{0},q_{2}) = (q_{0},\ell^{2})$. For generic data, the conic $q_0=0$ is smooth and the line $\ell=0$ meets it transversely at two distinct points. These are open conditions and showing that this happens for one specialization is enough to prove this generically.  Indeed, taking $A=I_{2}, b=(1,0), d=1,e=0,t=2$ results in $q_{0} = 1-x^{2}-y^{2}, q_{2} = 1-2x^{2}-y^{2}, \ell=-x$. Plugging $x=0$ into $q_{0}=0$ gives the equation $y=\pm 1$ hence nonzero discriminant. Therefore $q_{0}=q_{2}=0$ has two distinct points of intersection, for generic data.

Now let $p$ be such that $q_{0}(p)=q_{2}(p)=0$. WLOG $p=0$ by change of variables. Write $q_{0} = L + Q$ where $L\ne 0,Q\ne 0$ are homogeneous linear and quadratic respectively. $p=0$ is also a root of $\ell$ whence $\ell$ is homogeneous of degree one. Thus $\ell^{2}$ is homogeneous of degree two. This means $q_{2} = \kappa L + (\kappa Q - \Delta\ell^{2})$ where the terms of homogeneous degree two are in the parentheses. Therefore, both the conics $q_{0},q_{2}$ are smooth at $p$ and have the same tangent line $L=0$. Hence their intersection multiplicity at $p$ is at least two.

Since, generically, there are two distinct roots and their multiplicities sum to at most four (by Bezout's theorem), the intersection multiplicity at each distinct root is exactly two.
\end{proof}

\begin{lemma}\label{lem:no-triple-intersect}
$V(q_{0},q_{1},q_{2}) = \varnothing$ generically. 
\end{lemma}
\begin{proof}
It suffices to show that the homogenized conics have no common projective zero. This is a Zariski open condition and can be checked with a specialization. Use \[\Sigma = \begin{bmatrix}
-1&1&2&3&x&4\\
1&-2&5&6&y&7\\
2&5&-3&8&9&\alpha\\
3&6&8&-4&10&\beta\\
x&y&9&10&-5&11\\
4&7&\alpha&\beta&11&-6
\end{bmatrix}.\]

The following calculation uses the Macaulay2 package \texttt{Resultants}~\cite{Stagliano2018}.
\begin{lstlisting}[language=Macaulay2]
needsPackage "Resultants";
R = QQ[x,y,z];
Sigma0 = matrix {
   {-1, 1, x},
   {1, -2, y},
   {x, y, -5}
};
Sigma1 = matrix{
   {-1, 1, 2, 3, x},
   {1, -2, 5, 6, y},
   {2, 5, -3, 8, 9},
   {3, 6, 8, -4, 10},
   {x, y, 9, 10, -5}
};  
Sigma2 = matrix{
   {-1, 1, x, 4},
   {1, -2, y, 7},
   {x, y, -5, 11},
   {4, 7, 11, -6}
};
Q0 = homogenize(det Sigma0, z);
Q1 = homogenize(det Sigma1, z);
Q2 = homogenize(det Sigma2, z);
rho = resultant matrix{{Q0,Q1,Q2}};
rho != 0
\end{lstlisting}
The output of the last line is \texttt{true}.
\end{proof}

We still need to discount the roots which occur at infinity and those roots of the numerator which are also roots of the denominator.

\begin{lemma}
For generic observed data, the homogenizations of $F_{x}$ and $F_{y}$ have no common zero on the line at infinity.
\end{lemma}
\begin{proof}
Since $q_{i}$ are quadratics, let $Q_{i}(x,y)$ be the degree two homogeneous part of $q_{i}(x,y)$. Let the additional projectivization variable be $w$, to get to a degree five homogeneous polynomial. The roots at infinity would have $w=0$ and the equations are
\[F_{z}^{\infty}\sett Q_{0}Q_{2}\partial_{z} Q_{1} + Q_{0}Q_{1}\partial_{z}Q_{2} - Q_{1}Q_{2}\partial_{z} Q_{0} = 0, \qquad z\in \set{x,y}.\]
Euler's theorem gives $x\partial_{x}Q_{i}+y\partial_{y}Q_{i}=2Q_{i}\fa i\in \set{0,1,2}$ since $Q_{i}$ is homogeneous. This means $xF_{x}^{\infty}+yF_{y}^{\infty} = 2Q_{0}Q_{1}Q_{2}$. For generic data, $Q_{0},Q_{1},Q_{2}$ are squarefree and pairwise coprime. These open conditions hold for the specialization mentioned in \Cref{lem:no-triple-intersect} where
\[
\begin{aligned}
Q_0 &= 2x^2+2xy+y^2,\\
Q_1 &= -792x^2+516xy-191y^2,\\
Q_2 &= 37x^2-68xy+10y^2.
\end{aligned}
\]
After setting $y=1$, their discriminants are respectively $-4,-338832,3144$, and the resultants of $(Q_0,Q_1),(Q_0,Q_2),(Q_1,Q_2)$ are respectively $2517268,18769, 272860201$. All are nonzero, and none of the forms vanishes at $[1:0]$. If $F_{x}^{\infty} = F_{y}^{\infty} = 0$ then some $Q_{i}=0$ whence $\partial_{z}Q_{i}=0\fa z\in\set{x,y}$, contradicting squarefreeness of $Q_{i}$. Hence there is no common zero at infinity.
\end{proof}

Therefore we now have $25$ roots of $\Omega$ in the affine plane. It remains to discount the roots at the poles. The denominator can vanish if at least one $q_{i}$ vanishes. Suppose $p \in V(q_{0})\smallsetminus V(q_{1}q_{2})$. Then $\Omega(p) = -q_{1}(p)q_{2}(p)\d q_{0}(p)$. By genericity, $dq_{0}(p)\ne 0$. Therefore, $p$ is a non-root of $\Omega$. A similar argument holds if the conics are interchanged. So there is no root at a pole where only one $q_{i}$ vanishes. Now suppose $p$ is an intersection of two of these conics, say $p\in V(q_{0},q_{1})\smallsetminus V(q_{2})$. Then $\Omega(p) = 0$. Same argument holds for the other pairs of conics. There is no triple intersection. Therefore, we need to discount the roots of $\Omega=0$ at points where pairs among the curves $\set{q_{i}=0}_{i=0,1,2}$ intersect.

\begin{itemize}[leftmargin=*]
\item \textbf{Intersection of $q_{0}=q_{1}=0$.}  They meet transversally. One can use local coordinates $u=q_{0}, v=q_{1}$. Let $q_{2} = g(u,v)$ and this is a unit at $(0,0)$ since there is no triple intersection. So $\Omega/g = -v\d u + u\d v+uv\d \log g$. Say $\d \log g = a(u,v)\d u + b(u,v)\d v$. The coefficient of $\d u, \d v$ respectively are $A=auv-v = v(au-1), B=buv+u = u(bv+1)$. The intersection multiplicity is $\dim_{\C}\C[[u,v]]/(A,B)$. Both $au-1,bv+1$ are units in $\C[[u,v]]$ because their constant terms are units in $\C$ whence $(A,B) = (u,v)$ and thus the multiplicity is one at every intersection point. This contributes four roots because there are four simple intersections of $q_{0}=q_{1}=0$.
\item \textbf{Intersection of $q_{1}=q_{2}=0$.} The same argument as above lets us discount four simple roots of the numerator at the poles of the rational likelihood equations.
\item \textbf{Intersection of $q_{0}=q_{2}=0$.} This is the interesting case. We know these meet at two distinct points with multiplicity two each. Focus on one point and translate it to the origin. Recall from the proof of \Cref{cl:tangent} that $q_{2} = \kappa (q_{0} - v^{2})$, where $v=\ell\sqrt{\Delta/\kappa}$, on the open set $\Delta\kappa\ne 0$. These give local coordinates $q_{0}=u, q_{2} = \kappa(u-v^{2})$. $q_{1}$ is a unit because there is no triple intersection. Take $\tilde g = \kappa q_{1}$. Then $\Omega/\tilde g = u(u-v^{2}) \d\log \tilde g + v^{2} \d u - 2uv\d v$. Say $\d \log \tilde g = a(u,v) \d u + b(u,v)\d v$. Let $a_{0} \sett a(0,0), b_{0} \sett b(0,0)$. The coefficient of $\d u, \d v$ respectively are $A = v^{2} +au^{2}-auv^{2}, B = -2uv + bu^{2} - buv^{2}$. Therefore the intersection multiplicity is that of the origin of the curves $A=B=0$. Work over their completed coordinate ring $R = \hat\cO_{\mathbb A^{2},0} = \C[[u,v]]$. We want to find $\dim_{\C}R/(A,B)$. Let $f=v^{2}+a_{0}u^{2}, g=-2uv+b_{0}u^{2}$ be their respective homogeneous quadratic parts. They are coprime in $\C[[u,v]]$ precisely when $4a_{0} + b_{0}^{2} \ne 0$. A common projective zero must have $u\ne0$ and $v=b_{0}u/2$, so it exists precisely when $4a_{0}+b_{0}^{2}=0$. This is an open condition and direct computation with the specialization in \Cref{lem:no-triple-intersect} verifies $4a_{0}+b_{0}^{2}\ne 0$ at both tangencies, so this holds generically. We will show that $(A,B) = (f,g)$ in $R$. Note that $R$ is a local ring with maximal ideal $\mathfrak m = (u,v)$. Consider $J\sett (f,g), I=(A,B)$. Since $f,g$ are coprime and the binary homogeneous cubics form a four-dimensional space, $\mathfrak m J = \mathfrak m^{3}$. Indeed, $uf,ug,vf,vg$ is a basis of the space of binary homogeneous cubics. But $A-f,B-g\in \mathfrak m^{3} = \mathfrak mJ \subseteq J$. Therefore $\begin{bmatrix}A\\B\end{bmatrix} = (I_{2}+M)\begin{bmatrix}f\\g\end{bmatrix}$ where $M$ has entries in $\mathfrak m$. So $I\subseteq J$. Additionally, $\det (I_{2}+M) \in 1+\mathfrak m$ whence this matrix is invertible. This shows $J \subseteq I$.  Therefore the required multiplicity is $\dim_{\C}\C[[u,v]]/(A,B) = \dim_{\C}\C[[u,v]]/(f,g) = 4$ by \cite[Section~3.3, Theorem~3]{fulton_book_curves}.
\end{itemize}

$F_{x}, F_{y}$ are quintics. Projectivizing them gives $25$ roots by Bezout's theorem. Counted with multiplicity, there are $4+4+2\times 4=16$ zeroes of the numerator which are also poles. Therefore there are a total of $25-16=9$ roots of the numerator of the likelihood equations away from its poles. 
 
As a final verification step, the following rational specialization returns a zero-dimensional ideal of degree $9$ and verifies the above.

\begin{lstlisting}[language=Macaulay2]
R = QQ[x15,x25,x36,x46];
Sigma = matrix {
  {-1, 1,  2,  3,  x15, 4},
  {1,  -2, 5,  6,  x25, 7},
  {2,  5,  -3, 8,  9, x36},
  {3,  6,  8,  -4, 10, x46},
  {x15, x25, 9, 10, -5, 11},
  {4,  7,  x36, x46, 11, -6}
};
D = det Sigma;
cofactorNumerator = (M,i,j) -> (
  n := numRows M;
  rows := select(toList(0..(n-1)), r -> r != j);
  cols := select(toList(0..(n-1)), c -> c != i);
  (-1)^(i+j) * det submatrix(M,rows,cols)
);
missingEdges = {{0,4},{1,4},{2,5},{3,5}};
I0 = ideal apply(missingEdges, e -> cofactorNumerator(Sigma,e#0,e#1));
I = saturate(I0,ideal(D));
(dim I, degree I)
\end{lstlisting}

The output $(0,9)$ proves that this particular fiber has length $9$, counted with multiplicity, and agrees with \Cref{prop:double-blow-up-degree}.

\end{document}